\documentclass[a4paper,11pt]{article}
\usepackage{tipa}
\usepackage{bbm}
\usepackage{bbding}
\usepackage{stmaryrd}
\usepackage{amsfonts}
\usepackage{graphicx}
\usepackage{epsfig}
\usepackage{amssymb}
\usepackage{amssymb,amsmath,amsthm}
\usepackage{pifont}
\usepackage{graphicx}
\usepackage{cite}
\usepackage{xcolor}
\usepackage[breaklinks,colorlinks,linkcolor=red,urlcolor=blue,citecolor=cyan]{hyperref}

\DeclareMathOperator{\dive}{div}

\newtheorem{theorem}{Theorem}[section]
\newtheorem{lemma}[theorem]{Lemma}

\theoremstyle{definition}
\newtheorem{definition}{Definition}[section]

\newtheorem{remark}{Remark}
\newtheorem{Proposition}{Proposition}[section]

\allowdisplaybreaks[4]

\numberwithin{equation}{section}

\DeclareSymbolFontAlphabet{\testi}{letters} \large\normalsize

\begin{document}

\title{Global well-posedness for the 2D incompressible heat conducting Navier-Stokes equations with temperature-dependent coefficients and vacuum}

\author{Wenchao Dong$^2$, \ \ \ Qingyan Li$^{1,}$\thanks{Corresponding author. \newline \indent  \ \ \  Email addresses: wcdong@stumail.nwu.edu.cn (W.Dong), qyli22@126.com (Q. Li).} \vspace{0.13cm} \\
\textit{\small 1. School of Sciences, Chang'an University, Xi'an 710064, China;} \\
\textit{\small 2. Center for Nonlinear Studies, School of Mathematics, Northwest University, Xi'an 710127, China}}

\date{}
\maketitle{

\vspace{-0.85cm}

{\small
\noindent{\bf Abstract:} We consider the initial boundary problem of 2D non-homogeneous incompressible heat conducting Navier-Stokes equations with vacuum, where the viscosity and heat conductivity depend on temperature in a power law of Chapman-Enskog. We derive the global existence of strong solution to the initial-boundary value problem, which is not trivial, especially for the nonisentropic system with vacuum. Significantly, our  existence result holds for the cases that the viscosity and heat conductivity depend on $\theta$  with possibly different power laws (i.e., $\mu=\theta^{\alpha}, \kappa=\theta^{\beta}$ with constants $ \alpha,\beta\geq0$)  with   smallness assumptions only  on $\alpha$ and the measure of initial vacuum domain. In particular, the initial data can be arbitrarily large.  Moreover,  it is obtained that both velocity and temperature decay exponentially as time tends to infinity.

\vspace{0.1cm}

}
}

\vspace{0.1cm}


\date{}

\section{ Introduction}

The non-homogeneous incompressible heat conducting Navier-Stokes equations (see\cite{Lions} p.117, \cite{Lukaszewicz} p.23) can be stated as follows:
\begin{align}\label{1}
\left\{\begin{array}{l}
 \rho_t+\dive(\rho u)=0,\\
 (\rho u)_t+\dive(\rho u\otimes u)+\nabla P=\dive(2\mu D(u)),\\
 c_v\big((\rho \theta)_t+\dive(\rho u\theta)\big)-2\mu|D(u)|^2=\dive(\kappa\nabla\theta),\\
 \dive u=0.
\end{array}\right.
\end{align}
The system is supplemented with the initial condition
\begin{align}\label{2}
  (\rho, u, \theta)(x,0)=(\rho_0, u_0, \theta_0)(x), ~~ x\in \Omega,
\end{align}
and the boundary condition
\begin{align}\label{3}
  u=0, ~~ \frac{\partial \theta}{\partial\mathbf{n}}=0, ~~ \text{on} ~\partial\Omega,
\end{align}
where $\Omega\subset \mathbb{R}^2$ is a bounded smooth domain, $\mathbf{n}$ is the unit outward normal to $\partial\Omega$. Here $\rho, u, \theta$ and $P$ stand for the unknown density, velocity, absolute temperature and pressure of the fluid respectively. $D(u)=\frac{1}{2}\big(\nabla u+(\nabla u)^T\big)$ is the deformation tensor. The coefficients $\mu$, $c_v$ and $\kappa$ denote the viscosity, specific heat at constant volume and heat conductivity respectively.

In the theory of gas dynamics, the time evolution of the particle distribution function for the charged particles in a dilute gas can be modeled by the Vlasov-Poisson-Boltzmann system. It is well known that, if one derives the Navier-Stokes equations from the Boltzmann equation by applying the Chapman-Enskog expansion, the viscosity and heat conductivity coefficients are shown to be functions of absolute temperature in a power law. For details see \cite{Chapman,GuoYan}. If the intermolecular potential varies as $r^{-a}$, where $r$ means intermolecular distance, then $\mu$ and $\kappa$ are both proportional to a certain power of the temperature:
\begin{align*}
  \mu, \kappa\varpropto\theta^{\frac{a+4}{2a}}.
\end{align*}
For Maxwellian molecules ($a = 4$) the dependence is linear, while for elastic spheres ($a=\infty$)
the dependence is like $\sqrt{\theta}$. According to \cite{Chapman}, the following relations hold:
\begin{align*}
  \mu=c_1\theta^b, ~~~~
  \kappa=c_2\theta^b, ~~~~ b\in[\frac{1}{2}, \infty),
\end{align*}
where $c_1$ and $c_2$ are two positive constants. Based on this, in this paper, we mainly concentrate on the case that $c_v=1$ and $\mu, \kappa$ satisfy the following physical restrictions
\begin{equation}\label{4}
\mu=\mu(\theta)=\theta^{\alpha}, \ \ \kappa=\kappa(\theta)=\theta^{\beta},  \ \ \forall~  \alpha,\beta\geq0.
\end{equation}

Without heat-conducting, there are lots of literatures on the existence of solution to the $N$-dimensional ($N=2,3$) nonhomogeneous incompressible Navier-Stokes equations. If $\mu=const.$, the existence and uniqueness of solution are studied by many authors such as Danchin, Kim, Simon et al. in \cite{Cho2003,Kim1987,Simon,Craig,Danchin,Kim,Li,Liang,Lv}. For the case that $\mu=\mu(\rho)\geq C^{-1}$ is non-degenerate, many important and interesting phenomena in geophysical flows were modeled, as Liu-Xin-Yang \cite{Liutping} introduced the modified Navier-Stokes equations with density-dependent viscosity coefficients and point out that the viscosity depends only on the density for isentropic cases.
Cho and Kim\cite{Cho2004} first established the local existence of unique strong solution for the initial boundary value problem, then the authors \cite{Huang2,Huang1,ZhangJianwen} studied the global well-posedness to the multi-dimensional nonhomogeneous Navier-Stokes equations under the initial velocity is suitably small in certain sense. Recently, He, Li and L\"{u}\cite{He} obtained both the global existence and exponential stability of strong solution in $\mathbb{R}^3$, provided that the initial velocity is suitably small in some homogeneous Sobolev space. Meanwhile, for other related topic, there are also very interesting investigations about the existence of strong solution with small critical norms, refer to \cite{Abidi2011,Abidi2012,Abidi2013,Huang,Paicu,Paicu1013}, and references therein.

If the temperature equation $\eqref{1}_3$ is under consideration, the situation is more complicated and closer to the real status. In 2006, Feireisl and M\'{a}lek\cite{Feireisl2006} established long-time and large-data existence of a weak solution to the unsteady flows \eqref{1}, provided $C^{-1}\leq\mu(\theta),\kappa(\theta)\leq C$. Then, the local existence and uniqueness of the strong solution containing vacuum to the Cauchy problem \eqref{1}\eqref{2} in $\mathbb{R}^3$, for the case of $0<\mu(\rho,\rho\theta),\kappa(\rho,\rho\theta)\in C^1(\mathbb{R}^2)$, have been proved by Cho and Kim \cite{Cho}. Base on this, some global results have been obtained. When $\mu,\kappa=const.$, the authors\cite{Wang,Zhang,Zhong2020,Zhong,ZhongJMFM} considered 3D global strong solution  for different problems with small initial data, and showed that $u,P$ decay exponential and $\nabla\theta$ is algebraic.
When $\kappa$ is a positive constant, $\mu$ is density-temperature dependent and have positive lower bounds, the model is more practical. And Xu-Yu\cite{Xu,Xu2018} concerned the 3D Cauchy problem and the initial boundary value problem with vacuum, they obtained that $\Vert(\nabla u, \nabla\theta)\Vert_{H^1}\leq Ct^{-1}$ under the initial kinetic energy $\Vert\sqrt{\rho_0}u_0\Vert_{L^2}\ll1$.
Meanwhile, for 2D Cauchy problem and the initial boundary value problem when $\mu,\kappa= const.$ or density-dependent, the global existence and uniqueness of strong solution are established in \cite{Zhong2018,Zhong2022} by Zhong. However, these results do not include the decay estimate of the strong solution. Also, although the transport coefficients $\mu,\kappa$ may depend on $\theta$, they impose the condition that either or both $\mu$ and $\kappa$ are non-degenerate which have positive lower bounds. Nevertheless, in many applied sciences, Navier-Stokes equations with temperature-dependent transport coefficients are more practical, such as gases at very high temperature.
Recently, Guo and Li\cite{Guo} studied the 3D problem \eqref{1}-\eqref{4} under the small initial data  in the absence of vacuum, and got the large-time behavior of $(u,\nabla\theta)$.

The main aim of this paper is to study the global well-posedness of strong solution allowing vacuum to \eqref{1}-\eqref{4} in  two dimensional space without any smallness conditions about initial data. At the same time, we also prove that $\Vert( u,\theta-\frac{1}{\overline{\rho_0}|\Omega|}E_0)\big\Vert$ has a decay rate, while the density $\rho$ never decays (see Theorem \ref{T1}).

Without loss of generality, we denote
\begin{align*}
  \int\cdot dx\triangleq\int_{\Omega}\cdot dx,
\end{align*}
and use the following simplified notations for the standard Sobolev spaces:
\begin{align*}
  W^{k,p}=W^{k,p}(\Omega), ~~~~  H^k=H^k(\Omega)=W^{k,2}(\Omega), ~~~~ L^p=L^p(\Omega)=W^{0,p}(\Omega),\\
H_0^1=\{u\in H^1~ \big| ~ u=0\ \text{on}\ \partial\Omega\}, ~~~~~~~~ \,
H_{0,\sigma}^1=\{u\in H_0^1~ \big| ~ \dive u=0\  \text{in}\ \Omega\}, \\
H_{\mathbf{n}}^2=\{u\in H^2~ \big| ~ \nabla u\cdot\mathbf{n}=0\ \text{on}\ \partial\Omega\}, ~~~~
\Vert(f_1, f_2, \cdots, f_n)\Vert_{L^p}\triangleq\sum\limits_{i=1}^{n}\Vert f_i\Vert_{L^p},
\end{align*}
where $k\geq0, 1\leq p\leq\infty$.

Then, we give the definition of the strong solution to \eqref{1}-\eqref{4} throughout this paper.

\begin{definition}
(\emph{Strong Solution}) For $T>0$, $(\rho, u, \theta, P)$ is called a strong solution to \eqref{1}-\eqref{4} in $\Omega\times[0,T]$, if for some $q\in(2,\infty)$,
\begin{align}\label{45}
\left\{\begin{array}{l}
 \rho\in C([0,T];W^{1,q}),  ~~~ \rho_t\in C([0,T];L^q),\\
 u\in C([0,T];H_{0,\sigma}^1\cap H^2)\cap L^2(0,T;H^3),\\
 \theta\in C([0,T];H_{\mathbf{n}}^2)\cap L^2(0,T;H^3),\\
 P\in C([0,T];H^1)\cap L^2(0,T;H^2),\\
 (u_t,\theta_t)\in L^2(0,T;H^1), ~~~ (\sqrt{\rho}u_t,\sqrt{\rho}\theta_t)\in L^{\infty}(0,T;L^2),
\end{array}\right.
\end{align}
and $(\rho, u, \theta, P)$ satisfies \eqref{1} a.e. in $\Omega\times[0,T]$. In particular, the strong solution  $(\rho, u, \theta, P)$ is called the global strong solution, if the strong solution satisfies \eqref{45} for any $T>0$, and satisfies \eqref{1} a.e. in $\Omega\times[0,\infty)$.
\end{definition}

We are now in a position to state our main theorem as follows:

\begin{theorem}\label{T1}
For some positive constants $q\in(2,\infty)$ and $\underline{\theta}$, suppose that the initial data $(\rho_0, u_0, \theta_0)$ satisfies
\begin{equation}\label{49}
 0\leq\rho_0\in W^{1,q}, ~~~~
 u_0\in H_{0,\sigma}^1\cap H^2, ~~~~
 \underline{\theta}\leq\theta_0\in H_{\mathbf{n}}^2,
\end{equation}
and the compatibility conditions
\begin{align}\label{62}
  \left\{
    \begin{array}{rr}
     -\dive(2\mu(\theta_0)D(u_0))+\nabla P_0=\sqrt{\rho_0}g_1, \\
     -\dive(\kappa(\theta_0)\nabla\theta_0)-2\mu(\theta_0)|D(u_0)|^2=\sqrt{\rho_0}g_2,
    \end{array}
  \right.
\end{align}
for some $P_0\in H^1$ and $g_1,g_2\in L^2$. Then there exist two positive constants $\epsilon_0$ and $c_0$ depending only on $\tilde{\rho}$ and $\Vert(u_0,\theta_0)\Vert_{H^2}$,  such that if
\begin{align}
 \alpha\leq\epsilon_0
\end{align}
and
\begin{align}\label{79}
 |V|\leq \exp(-\frac{1}{c_0^2}),
\end{align}
where $|V|$ is the measure of the initial domain $V=\{x\in\Omega ~| ~\rho_0(x)\leq c_0\}$, the initial boundary problem \eqref{1}-\eqref{4} admits a unique global strong solution $(\rho,u,\theta,P)$ for any $T>0$, and the following large-time behavior holds:
\begin{equation}\label{47}
  \lim_{t\rightarrow\infty}
  \big\Vert( u, ~
   \theta-\frac{1}{\overline{\rho_0}|\Omega|}E_0)\big\Vert_{L^{\infty}}
  =0.
\end{equation}
Particularly, $(u,\theta,P)$ has the following decay rates
\begin{align}
  \Vert u\Vert_{H^2}^2+\Vert P\Vert_{H^1}^2+\Vert\sqrt{\rho} u_t\Vert_{L^2}^2
  \leq C_0e^{-\sigma_1 t},\label{54}\\
  \Vert\theta-\frac{1}{\overline{\rho_0}|\Omega|}E_0\Vert_{H^2}^2
  +\Vert\sqrt{\rho}\theta_t\Vert_{L^2}^2\leq C_0e^{-\sigma_2 t},\label{55}
\end{align}
where $C_0$  is a positive constant depending only on $c_0, \tilde{\rho}, \underline{\theta}, \Omega, \Vert(u_0,\theta_0)\Vert_{H^2}$ and $\Vert\rho_0\Vert_{W^{1,q}}$. Here
\begin{gather*}
  E_0=\int\rho_0(\theta_0+\frac{1}{2}|u_0|^2)dx, ~~~~~~~
  \overline{\rho_0}=\frac{1}{|\Omega|}\int\rho_0dx, ~~~~~~~
  \sigma_1\triangleq\frac{\pi^2\underline{\theta}^{\alpha}}{\tilde{\rho}d^2},\\
  \sigma_2\triangleq\frac{\pi^2}{\tilde{\rho}d^2}
     \min\big\{\frac{1}{2}\underline{\theta}^{\beta}(1+\frac{\tilde{\rho}}{\overline{\rho_0}})^{-2}, \underline{\theta}^{\alpha}\big\}, ~~~~
     d=\text{diam}(\Omega)\triangleq\sup\{|x-y| | x,y\in\Omega\}, ~~~~ \tilde{\rho}=\Vert\rho_0\Vert_{L^{\infty}}.
 \end{gather*}

\end{theorem}

 \begin{remark}
 In the special case that   the initial density does not contain vacuum, i.e., $\rho_0\geq \underline{\rho}>0$, it is clear that conditions \eqref{62} and \eqref{79} are naturally satisfied since  we can  take $c_0=\frac{1}{2}\underline{\rho}$. Therefore, the conclusion in our Theorem \ref{T1} holds directly for the case of initial density without vacuum.
 \end{remark}

\begin{remark}
From the subsequent proofs \eqref{27}-\eqref{64}, we find that condition \eqref{79},which is inspired by the literature \cite{Cao}, can be reduced to
\begin{itemize}
  \item there exists a positive constant $c_0$ such that
    \begin{align}\label{83}
      |V|\leq\frac{\exp(-\frac{C(2\beta+2)}{c_0})}{64C^{2\beta+4}},
    \end{align}
  where $C$ is a positive constant defined in \eqref{27} and depending only on $\tilde{\rho},\underline{\theta}, \Omega$ and $\Vert(u_0,\theta_0)\Vert_{H^2}$.
\end{itemize}
In fact, the conditions \eqref{79} and \eqref{83} are essentially similar in that they both indicate that: $(1)$ the measure of the initial vacuum domain is sufficiently small; $(2)$ $\rho_0$ near the vacuum grows to $c_0$ at a very fast rate (this growth rate depends on $\tilde{\rho},\underline{\theta}, \Omega$ and $\Vert(u_0,\theta_0)\Vert_{H^2}$).
\end{remark}

\begin{remark}
In order to understand  the condition \eqref{79}, we mention that the following class of initial data is a special case of \eqref{49}-\eqref{79} provided $c_0$ is  small enough.
Assume $\Omega=B_1(0)=\{x\in \mathbb{R}^2 ~| ~|x|\leq1\}$, for any $k_1, k_2\in(0,2)$ and $\epsilon=\pi^{-\frac{1}{2}}\exp(-\frac{1}{2c_0^2})$, let
\begin{align*}
  \rho_0(x)
  =\left\{
     \begin{array}{ll}
       ~~~ 0, & |x|\leq\frac{\epsilon}{2}\\
       c_0(\frac{2}{\epsilon}|x|-1)^{k_1}, & \frac{\epsilon}{2}\leq|x|\leq \epsilon \\
       (\frac{2(1-c_0^{\frac{1}{k_2}})}{\epsilon}|x|-2+3c_0^{\frac{1}{k_2}})^{k_2}, &
        \epsilon\leq|x|\leq \frac{3\epsilon}{2} \\
       ~~~ 1, & \frac{3\epsilon}{2}\leq|x|\leq1
     \end{array}
   \right.
\end{align*}
and $(u_0, \theta_0)\in C^2$ satisfy the initial regularity conditions \eqref{49} and the compatibility conditions \eqref{62}. Thus,  we can check that
\begin{align*}
  |V|=|\{x\in\Omega ~| ~\rho_0(x)\leq c_0\}|
  = \pi\epsilon^2
  =\exp(-\frac{1}{c_0^2}).
\end{align*}
\end{remark}

\begin{remark}
It seems that Theorem \ref{T1} is the first study concerning the 2D problem \eqref{1}-\eqref{4}  for arbitrarily large initial data with vacuum and temperature-dependent coefficients.   This is in sharp contrast to Zhong\cite{Zhong2018,Zhong2022} and Guo-Li\cite{Guo} where they need either $\mu,\kappa=const.$, or the smallness assumptions on both $\Vert\sqrt{\rho_0}u_0\Vert_{L^2}$ and $\Vert\rho_0\theta_0\Vert_{L^1}$ without initial vacuum.
  Besides the initial mass, velocity and temperature being arbitrarily large although $\alpha$ is small, $u$ and $\theta$ are all exponentially decaying as time tends to infinity. Moreover, it is easy to see that $\rho$ is not have any decay estimate due to $\dive u=0$ (see Remark \ref{r3} in section \ref{S5} for the details).
\end{remark}

We now make some comments on the analysis in this paper. To extend the local strong solution whose existence is obtained by Lemma \ref{L5.1} globally in time, one needs to establish global a priori estimates on smooth solution to \eqref{1}-\eqref{4} in suitable higher norms. There is extremely strong nonlinearity and degeneracy caused by transport coefficients, both of which create great difficulty for the a priori estimates, especially for the second-order estimates.   The main idea is to combine the bootstrap argument and time-weighted estimates successfully applied to the Navier-Stokes equations. It turns out that as in \cite{Guo}, the key ingredient here is to obtain the time-independent bounds on $\theta(x,t)$. However, the methods applied in \cite{Guo,Zhong2022} rely crucially on the smallness of the initial data or $\mu,\kappa=const.$. Therefore, some new ideas are needed here. First of all, according to the regularity properties of Stokes system and the smallness of $\alpha$, using an inequality derived by Desjardins\cite{Desjardins}(see Lemma \ref{L04}), we establish the a priori estimates of $\theta$ in suitable norms. Here, the main obstacle comes from the estimate of $\theta$ in vacuum, which strongly interacts on the velocity field. Motivated by \cite{Cao}, we divide the integration region into two parts (see \eqref{82}) and get the estimates of the temperature from the condition \eqref{79}. Then, by using the Poincar\'{e}'s inequality and carefully calculations, we derive the decay estimates of the solution $(u,\theta,P)$. Meanwhile, we obtain the higher-order estimates of the solution owing to $L^2$-theory of elliptic equations. Finally, applying these a priori estimates and the fact that the velocity is divergent free, we can extend the local strong solution globally in time.

The rest of the paper is organized as follows: In section \ref{S3}, we present some basic facts and inequalities which will be used later. Section \ref{S4} concerns some a priori estimates on smooth solution which are needed to extend the local solution to all time. Finally,  with all a priori estimates at hand, the main result Theorem \ref{T1} is proved in section \ref{S5}.


\section{ Preliminaries}\label{S3}

In this section we shall enumerate some auxiliary lemmas used in this paper.
We first give the famous  Gr\"{o}nwall's inequality\cite{HAmann} which will play an essential role in the energy estimate of $(\rho,u,\theta)$.
\begin{lemma}\label{L01}
Suppose that $f_1(t),f_2(t): [0,T]\rightarrow \mathbb{R}$ are nonnegative bounded measurable function,
$c(t): [0,T]\rightarrow \mathbb{R}$ is a nonnegative integrable function. If $f_1(t), f_2(t), c(t)$ satisfy
\begin{equation*}
  f_1(t)\leq f_2(t)+\int_{0}^{t}c(s)f_1(s)ds, ~~~ \forall ~ t\in[0,T],
\end{equation*}
then it follows that
\begin{equation*}
  f_1(t)\leq f_2(t)+\int_{0}^{t}f_2(s)c(s)\exp\big(\int_{s}^{t}c(\tau)d\tau\big)ds, ~~~ \forall ~ t\in[0,T].
\end{equation*}
Moreover, if $f_2(t)$ is a monotone increasing function over $[0,T]$, then we obtain the estimate
\begin{equation*}
  f_1(t)\leq f_2(t)\exp\big(\int_{0}^{t}c(s)ds\big), ~~~ \forall ~ t\in[0,T].
\end{equation*}
\end{lemma}

The following Bihari-LaSalle inequality (\cite{Bihari1956,Dhongade1976,LaSalle1949}) is a nonlinear generalization of Gr\"{o}nwall's inequality.

\begin{lemma}\label{LA1}
  Suppose that
\begin{itemize}
  \item $y(t)\geq 0$, \ $0\leq h(t)\in L^1(0,T)$,
  \item $0<w(y)$ is continuous and nondecreasing for $y>0$,
  \item $c_1, c_2$ are two positive constants.
\end{itemize}

If
\begin{align*}
  y(t)\leq c_1+c_2\int_0^th(s)w\big(y(s)\big)ds,  ~~~~~ \forall \, t\in[0,T],
\end{align*}
then
\begin{align}\label{81}
  y(t)\leq G^{-1}\Big(G(c_1)+c_2\int_0^th(s)ds\Big), ~~~~~ \forall \, t\in[0,T],
\end{align}
where
\begin{align*}
  G(x)=\int_{x_0}^x\frac{1}{w(y)}dy, ~~~~~ x\geq0,\, x_0>0,
\end{align*}
and $G^{-1}$ is the inverse of $G$, $T$ is chosen so that the right hand of \eqref{81} is well-defined.
\end{lemma}

Because the initial data contains a vacuum, there are a lot of places that need to use the following Poincar\'{e} type inequality.

\begin{lemma}\label{L08}
Let $f\in H^1(\Omega)$, and $0\leq g\leq c_1$, $\int gdx\geq c_1^{-1}$. Then there exists a positive constant $C$ depending only on $c_1, p, \Omega$ such that
\begin{align*}
  \Vert f\Vert_{L^p}
  \leq C\Vert gf\Vert_{L^1}+C\Vert\nabla f\Vert_{L^2},  ~~~ \forall \, p\geq1.
\end{align*}
\end{lemma}

\begin{proof}
It follows from Poincar\'{e}'s inequality that
  \begin{align*}
    \big|\overline{g}\int f dx\big|
    =\big|\int g dx\overline{f}\big|
    =\big|\int g fdx+\int g(\overline{f}- f)dx\big|
    \leq \Vert gf\Vert_{L^1}+C\Vert\nabla f\Vert_{L^2},
  \end{align*}
where $\overline{g}=\frac{1}{|\Omega|}\int gdx$. Thus, $\forall \, p\geq1$,
\begin{align*}
  \Vert f\Vert_{L^p}
  &\leq\Vert f-\overline{f}\Vert_{L^p}+\Vert\overline{f}\Vert_{L^p}\\
  &\leq C\Vert\nabla f\Vert_{L^2}+C|\overline{f}|\\
  &\leq C\Vert gf\Vert_{L^1}+C\Vert\nabla f\Vert_{L^2}.
\end{align*}
We complete the proof of this lemma.
\end{proof}

Next, we present Gagliardo-Nirenberg inequality (\cite{Galdi,Nirenberg}) that is frequently employed in our proof.
\begin{lemma}\label{L02}
Let $u\in L^n(\Omega)\cap L^{\tilde{n}}(\Omega)$,
with $\nabla^iu\in L^m(\Omega)$, $i>0$, $\tilde{n}>0$, $1\leq m, n\leq \infty$. Then, $\nabla^ju\in L^k(\Omega)$ and the following inequality holds
for $0\leq j<i$ and some $C= C( i, j, m, n, \gamma, \Omega)$:
\begin{align*}
  \Vert\nabla^ju\Vert_{L^k}\leq C\big(\Vert\nabla^iu\Vert_{L^m}^\gamma\Vert u\Vert_{L^n}^{1-\gamma}
   +\Vert u\Vert_{L^{\tilde{n}}}\big),
\end{align*}
where
\begin{align*}
  \frac{1}{k}-\frac{j}{2}=(\frac{1}{m}-\frac{i}{2})\gamma+\frac{1}{n}(1-\gamma),
\end{align*}
for all $\gamma$ in the interval $\frac{j}{i}\leq \gamma \leq 1$. If $1<m<\infty$ and $i-j-\frac{2}{m}$ is a non-negative integer, then it is necessary to assume also that $\gamma\neq1$.
\end{lemma}

We then state some elementary estimates for the following nonhomogeneous Stokes equations, which is used to be get the derivations of high order estimates of $u$:
\begin{align}\label{5}
\left\{\begin{array}{l}
  -\dive\big(2\mu(\theta)D(u)\big)+\nabla P=F,  ~~ \text{in}~\Omega,\\
  \dive u=0,  ~~~~~~~~~~~~~~~~~~~~~~~~~~~~\, \text{in}~\Omega,\\
  u=0,  ~~~~~~~~~~~~~~~~~~~~~~~~~~~~~~~~~ \text{on}~\partial\Omega,\\[1mm]
  \displaystyle\int\frac{P}{\mu(\theta)}dx=0.
\end{array}\right.
\end{align}

\begin{lemma}[\cite{Galdi,Huang1,Huang2}]\label{L03}
Assume that $\underline{\mu}\leq\mu(\theta)\leq\overline{\mu}$ and $\nabla\mu(\theta)\in L^k$ for some $k\in(2,\infty)$. Let $(u,P)\in H_0^1\times L^2$ be the unique weak solution to the problem \eqref{5}, then there exists a positive constant $C=C(k, \underline{\mu}, \overline{\mu},\Omega)$ such that the following regularity results hold true:
\begin{itemize}
  \item [(1)] If $F\in L^r$ for some $r\in[2,k)$, then $(u,P)\in W^{2,r}\times W^{1,r}$ and
  \begin{align*}
    \Vert u\Vert_{W^{2,r}}+\Vert \frac{P}{\mu(\theta)}\Vert_{W^{1,r}}
    \leq C\Vert F\Vert_{L^r}\big(1+\Vert\nabla\mu(\theta)\Vert_{L^k}^
    {\frac{k}{k-2}\cdot\frac{2r-2}{r}}\Big).
  \end{align*}
  \item [(2)] If $F\in H^1$ and $\nabla\mu(\theta)\in H^1$, then $(u,P)\in H^3\times H^2$ and
  \begin{align*}
    \Vert u\Vert_{H^3}+\Vert \frac{P}{\mu(\theta)}\Vert_{H^2}
    \leq C\Vert F\Vert_{H^1}(1+\Vert\nabla\mu(\theta)\Vert_{H^1}^{\frac{k}{k-2}+2}).
  \end{align*}
\end{itemize}
\end{lemma}
\begin{proof}
$\emph{(1)}$.
From \cite[Lemma 2.1]{Huang2},  we obtain that
\begin{align}\label{67}
  \Vert(\nabla u, \frac{P}{\mu(\theta)})\Vert_{L^2}
  \leq C\Vert F\Vert_{L^2}.
\end{align}
Then, the equation \eqref{5} can be rewritten as
\begin{align*}
  -\Delta u+\nabla\frac{P}{\mu(\theta)}
  =\frac{1}{\mu(\theta)}\big(F+2\nabla\mu(\theta)\cdot D(u)-\frac{P}{\mu(\theta)}\nabla\mu(\theta)\big),
\end{align*}
the classical theory for Stokes equations (\cite{Galdi,Huang1}) and Gagliardo-Nirenberg inequality (Lemma \ref{L02}) give that
\begin{align*}
  &\Vert u\Vert_{W^{2,r}}+\Vert\nabla\frac{P}{\mu(\theta)}\Vert_{L^r}\\
  &\leq C\Vert F\Vert_{L^r}+C\Vert\nabla\mu(\theta)\cdot D(u)\Vert_{L^r}
   +C\Vert\frac{P}{\mu(\theta)}\nabla\mu(\theta)\Vert_{L^r}\\
  &\leq C\Vert F\Vert_{L^r}+C\Vert\nabla\mu(\theta)\Vert_{L^k}
    \big(\Vert\nabla u\Vert_{L^{\frac{kr}{k-r}}}
   +\Vert\frac{P}{\mu(\theta)}\Vert_{L^{\frac{kr}{k-r}}}\big)\\
  &\leq C\Vert F\Vert_{L^r}+C\Vert\nabla\mu(\theta)\Vert_{L^k}
    \big(\Vert\nabla u\Vert_{L^2}^{\frac{kr-2r}{2kr-2k}}
    \Vert\nabla u\Vert_{W^{1,r}}^{\frac{kr-2k+2r}{2kr-2k}}
   +\Vert\frac{P}{\mu(\theta)}\Vert_{L^2}^{\frac{kr-2r}{2kr-2k}}
    \Vert\nabla\frac{P}{\mu(\theta)}\Vert_{L^r}^{\frac{kr-2k+2r}{2kr-2k}}\big)\\
  &\leq\frac{1}{2}\Vert\nabla u\Vert_{W^{1,r}}+\frac{1}{2}\Vert\nabla\frac{P}{\mu(\theta)}\Vert_{L^r}
   +C\Vert F\Vert_{L^r}+C\Vert\nabla\mu(\theta)\Vert_{L^k}^{\frac{2kr-2k}{kr-2r}}
   \Vert(\nabla u, \frac{P}{\mu(\theta)})\Vert_{L^2},
\end{align*}
which together with \eqref{67} yields
\begin{align}\label{68}
  \Vert u\Vert_{W^{2,r}}+\Vert\nabla\frac{P}{\mu(\theta)}\Vert_{L^r}
  \leq C\Vert F\Vert_{L^r}(1+\Vert\nabla\mu(\theta)\Vert_{L^k}^{\frac{k}{k-2}\frac{2r-2}{r}}).
\end{align}

$\emph{(2)}$. It follows from \eqref{5}, \eqref{68}, Gagliardo-Nirenberg inequality that
\begin{align*}
  &\Vert u\Vert_{H^3}+\Vert\frac{P}{\mu(\theta)}\Vert_{H^2}\\
  &\leq C\Vert\frac{1}{\mu(\theta)}\big(F+2\nabla\mu(\theta)\cdot D(u)
   -\frac{P}{\mu(\theta)}\nabla\mu(\theta)\big)\Vert_{H^1}\\
  &\leq C\Vert F\Vert_{L^2}(1+\Vert\nabla\mu(\theta)\Vert_{L^k}^{\frac{k}{k-2}})
   +C\Vert\nabla F\Vert_{L^2}
   +C\Vert F\nabla\mu(\theta)\Vert_{L^2}\\
  &\quad +C\Vert|\nabla\mu(\theta)|^2|\nabla u|\Vert_{L^2}
   +C\Vert|\nabla^2\mu(\theta)||\nabla u|\Vert_{L^2}
   +C\Vert|\nabla\mu(\theta)||\nabla^2 u|\Vert_{L^2}\\
  &\quad +C\Vert|\nabla\mu(\theta)|^2\frac{P}{\mu(\theta)}\Vert_{L^2}
   +C\Vert|\nabla\frac{P}{\mu(\theta)}||\nabla\mu(\theta)|\Vert_{L^2}
   +C\Vert\frac{P}{\mu(\theta)}|\nabla^2\mu(\theta)|\Vert_{L^2}\\
  &\leq C\Vert F\Vert_{L^2}(1+\Vert\nabla\mu(\theta)\Vert_{H^1}^{\frac{k}{k-2}})
   +C\Vert\nabla F\Vert_{L^2}
   +C\Vert F\Vert_{L^4}\Vert\nabla\mu(\theta)\Vert_{L^4}\\
  &\quad +C\Vert\nabla\mu(\theta)\Vert_{L^8}^2\Vert(\nabla u,\frac{P}{\mu(\theta)})\Vert_{L^4}
   +C\Vert\nabla^2\mu(\theta)\Vert_{L^2}\Vert(\nabla u,\frac{P}{\mu(\theta)})\Vert_{L^{\infty}}\\
  &\quad +C\Vert\nabla\mu(\theta)\Vert_{L^4}\Vert(\nabla^2 u,\nabla\frac{P}{\mu(\theta)})\Vert_{L^4}\\
  &\leq C\Vert F\Vert_{H^1}(1+\Vert\nabla\mu(\theta)\Vert_{H^1}^{\frac{k}{k-2}})
   +C\Vert\nabla\mu(\theta)\Vert_{H^1}^2\Vert(\nabla u,\frac{P}{\mu(\theta)})\Vert_{H^1}\\
  &\quad +C\Vert\nabla^2\mu(\theta)\Vert_{L^2}
    \Big(\Vert(\nabla u,\frac{P}{\mu(\theta)})\Vert_{L^2}^{\frac{1}{2}}
    \Vert(\nabla^3 u,\nabla^2\frac{P}{\mu(\theta)})\Vert_{L^2}^{\frac{1}{2}}
    +\Vert(\nabla u,\frac{P}{\mu(\theta)})\Vert_{L^2}\Big)\\
  &\quad +C\Vert\nabla\mu(\theta)\Vert_{H^1}
   \Big(\Vert(\nabla^2 u,\nabla\frac{P}{\mu(\theta)})\Vert_{L^2}^{\frac{1}{2}}
    \Vert(\nabla^3 u,\nabla^2\frac{P}{\mu(\theta)})\Vert_{L^2}^{\frac{1}{2}}
    +\Vert(\nabla^2 u,\nabla\frac{P}{\mu(\theta)})\Vert_{L^2}\Big)\\
  &\leq \frac{1}{2}\Vert(\nabla^3 u,\nabla^2\frac{P}{\mu(\theta)})\Vert_{L^2}
    +C\Vert F\Vert_{H^1}(1+\Vert\nabla\mu(\theta)\Vert_{H^1}^{\frac{k}{k-2}+2}).
\end{align*}
This completes the proof of Lemma \ref{L03}.
\end{proof}

Next, we give two inequalities that are extremely important for the estimate of $\nabla^2\theta$.

\begin{lemma}[\cite{Galdi}]\label{L05}
Let $v$ be a vector function with components in $W^{1,p}$, $p\in[1, \infty)$, and $v\cdot\mathbf{n}=0$ on $\partial\Omega$. Then
\begin{align*}
  \Vert v\Vert_{L^p}\leq C\Vert\nabla v\Vert_{L^p},
\end{align*}
where the constant $C$ depends only on $p$ and $\Omega$.
\end{lemma}

\begin{lemma}[\cite{Luotao,Amrouche}]\label{L06}
Suppose that $\theta\in H^{k+2}, k\geq0$ and $\frac{\partial\theta}{\partial\mathbf{n}}\big|_{\partial\Omega}=0$. Then, it holds that
\begin{align*}
  \Vert\nabla^2\theta\Vert_{H^k}\leq C\big(\Vert\Delta\theta\Vert_{H^k}
   +\Vert\nabla\theta\Vert_{L^2}\big).
\end{align*}
\end{lemma}

Finally, for $u\in H_0^1(\Omega)$, by the Gagliardo-Nirenberg's inequality, we have
\begin{align}\label{6}
  \Vert u\Vert_{L^4}\leq C\Vert u\Vert_{L^2}^{\frac{1}{2}}\Vert\nabla u\Vert_{L^2}^{\frac{1}{2}}.
\end{align}
However, to deal with a nonhomogeneous problem with vacuum, some interpolation inequality for $u$ with degenerate weight like $\sqrt{\rho}$ is required. We look for a similar estimate for $\sqrt{\rho}u$ as in \eqref{6}. By zero extension of $u$ outside the bounded domain $\Omega$, we can derive the following lemma first established by Desjardins \cite{Desjardins} which reads as follows.

\begin{lemma}\label{L04}
Suppose that $0\leq \rho\leq\tilde{\rho}$, $u\in H_0^1(\Omega)$, then we have
\begin{align*}
  \Vert\sqrt{\rho}u\Vert_{L^4}^2\leq C(\tilde{\rho},\Omega)(1+\Vert\sqrt{\rho}u\Vert_{L^2})
  \Vert\nabla u\Vert_{L^2}\sqrt{\log(2+\Vert\nabla u\Vert_{L^2}^2)}.
\end{align*}
\end{lemma}


\section{A Priori Estimates}\label{S4}

In the following sections, we denote
\begin{itemize}
  \item $C$ is a positive constant depending only on $\tilde{\rho},\underline{\theta}, \Omega$ and $\Vert(u_0,\theta_0)\Vert_{H^2}$, but independent of the time $T$ and $\rho_0$.
  \item $C_0$ is a positive constant depending on $\tilde{\rho}, \underline{\theta}, \Omega$, $\Vert(u_0,\theta_0)\Vert_{H^2}, \Vert\rho_0\Vert_{W^{1,q}}, c_0$ and the domain $V$, but independent of  $T$.
\end{itemize}

\subsection{Bootstrap Argument}

In this subsection, we will establish some necessary a priori estimates of the strong
solution $(\rho, u, \theta, P)$ to the problem \eqref{1}-\eqref{4}. Thus, let $T > 0$ be a fixed time and $(\rho, u, \theta, P)$ be the smooth solution to \eqref{1}-\eqref{4} on $\Omega\times[0,T]$ with smooth initial data $(\rho_0, u_0, \theta_0)$ satisfying \eqref{49}-\eqref{79}. Therefore, we have the following key a priori estimates on $(\rho, u, \theta, P)$.

\begin{Proposition}\label{P0}
There exist two positive constants $M$ and $\epsilon_0\ll1$ all depending only on the initial data such that if $(\rho, u, \theta, P)$ is a smooth solution of \eqref{1}-\eqref{4} on $\Omega\times[0,T]$ satisfying
\begin{align}\label{60}
  \sup_{0\leq t\leq T}\Vert\theta\Vert_{H^2}^2
  +\int_0^T\big((1+t)\Vert\theta_t\Vert_{L^2}^2
   +(1+t^2)\Vert\nabla\theta_t\Vert_{L^2}^2\big)dt
  \leq 2M,
\end{align}
then, the following estimate holds
\begin{align}\label{65}
  \sup_{0\leq t\leq T}\Vert\theta\Vert_{H^2}^2
  +\int_0^T\big((1+t)\Vert\theta_t\Vert_{L^2}^2
   +(1+t^2)\Vert\nabla\theta_t\Vert_{L^2}^2\big)dt
  \leq M,
\end{align}
provided $\alpha\leq\epsilon_0$.
\end{Proposition}

Before proving Proposition \ref{P0}, we establish some necessary a priori estimates, see Lemmas \ref{L1}-\ref{L5}.

\begin{lemma}\label{L1}
Under the assumptions of Proposition \ref{P0}, $\forall \, (x,t)\in\Omega\times[0,T]$, it holds that
\begin{gather}
  0\leq \rho(x,t)\leq \tilde{\rho},\label{7}\\
  \theta(x,t)\geq \underline{\theta}.\label{8}
\end{gather}
\end{lemma}
\begin{proof}
The proof of \eqref{7} is given by \cite{Feireisl,Guo,Lions}.
Then, applying standard maximum principle (\cite{Feireisl}, p.43) to $\eqref{1}_3$ along with $\theta_0\geq \underline{\theta}$ shows \eqref{8}.
\end{proof}

\begin{lemma}\label{L2}
Under the assumptions of Proposition \ref{P0}, it holds that
\begin{gather}\label{9}
  \sup_{0\leq t\leq T}\big((1+t^2)\Vert u\Vert_{H^1}^2\big)
  +\int_0^T(1+t^2)(\Vert\nabla u\Vert_{H^1}^2+\Vert\sqrt{\rho} u_t\Vert_{L^2}^2)dt
  \leq C.
\end{gather}
\end{lemma}
\begin{proof}
\underline{step 1}.
Multiplying $\eqref{1}_2$ by $2u$, then integrating the resulting identity over $\Omega$ yields
\begin{equation}\label{10}
  \frac{d}{dt}\int\rho|u|^2dx+\int4\mu(\theta)|D(u)|^2dx=0.
\end{equation}
And integrating \eqref{10}  with respect to $t$, we have
\begin{equation}\label{12}
  \Vert\sqrt{\rho}u\Vert_{L^2}^2
  +4\int_0^t\int\mu(\theta)|D(u)|^2dxds
  =\Vert\sqrt{\rho_0}u_0\Vert_{L^2}^2.
\end{equation}
Then, Multiplying \eqref{10} by $t$ and integrating it over $(0,t)$, we obtain from Poincar\'{e}'s inequality, Lemma \ref{L1} and \eqref{12} that
\begin{equation}\label{13}
  t\Vert\sqrt{\rho}u\Vert_{L^2}^2
  +\int_0^ts\int4\mu(\theta)|D(u)|^2dxds
  \leq\int_0^t\Vert\sqrt{\rho}u\Vert_{L^2}^2ds
  \leq C\int_0^t\Vert\nabla u\Vert_{L^2}^2ds
  \leq C.
\end{equation}
In a similar manner, one can get that
\begin{equation}\label{14}
  t^2\Vert\sqrt{\rho}u\Vert_{L^2}^2
  +\int_0^ts^2\int4\mu(\theta)|D(u)|^2dxds
  \leq C.
\end{equation}

\underline{step 2}.
Note that Lemma \ref{L1} and \eqref{60} tells us that
\begin{align*}
  C^{-1}\leq\underline{\theta}^{\alpha}\leq\mu(\theta)
  \leq\Vert\theta\Vert_{L^{\infty}}^{\alpha}
  \leq\Vert\theta\Vert_{H^2}^{\alpha}
  \leq(2M)^{\frac{\alpha}{2}}\leq C,
\end{align*}
provided
\begin{align}\label{17}
  \alpha\leq\min\{1, M^{-1}\}.
\end{align}
Recall that $(u,P)$ satisfies the following Stokes system:
\begin{align}\label{70}
\left\{\begin{array}{l}
  -\dive\big(2\mu(\theta)D(u)\big)+\nabla P=-\rho u_t-\rho u\cdot\nabla u,  ~~ \text{in}~\Omega,\\
  \dive u=0,  ~~~~~~~~~~~~~~~~~~~~~~~~~~~~~~~~~~~~~~~~~~~~~~ \, \text{in}~\Omega,\\
  u=0,  ~~~~~~~~~~~~~~~~~~~~~~~~~~~~~~~~~~~~~~~~~~~~~~~~~~~ \text{on}~\partial\Omega.
\end{array}\right.
\end{align}
Applying Lemma \ref{L03} with $F=-\rho u_t-\rho u\cdot\nabla u$, then using Gagliardo-Nirenberg inequality (Lemma \ref{L02}), \eqref{60} and \eqref{17}, we arrive at
\begin{align}\label{44}
    &\Vert u\Vert_{H^2}+\Vert \frac{P}{\mu(\theta)}\Vert_{H^1}\nonumber\\
    &\leq C\big(\Vert \rho u_t\Vert_{L^2}+\Vert\rho u\cdot\nabla u\Vert_{L^2}\big)
    \big(1+\Vert\alpha\theta^{\alpha-1}\nabla\theta\Vert_{L^k}^{\frac{k}{k-2}}\big)\nonumber\\
    &\leq C\big(\Vert\sqrt{\rho} u_t\Vert_{L^2}+\Vert\rho u\Vert_{L^4}\Vert\nabla u\Vert_{L^4}\big)
    \big(1+(\alpha M)^{\frac{k}{k-2}}\big)\nonumber\\
    &\leq C\big(\Vert\sqrt{\rho} u_t\Vert_{L^2}
    +\Vert\sqrt{\rho} u\Vert_{L^4}\Vert\nabla u\Vert_{L^2}^{\frac{1}{2}}
    \Vert\nabla u\Vert_{H^1}^{\frac{1}{2}}\big).
\end{align}
It follows from \eqref{44} that
\begin{align}\label{15}
    \Vert u\Vert_{H^2}+\Vert \frac{P}{\mu(\theta)}\Vert_{H^1}
    \leq C\big(\Vert\sqrt{\rho} u_t\Vert_{L^2}
    +\Vert\sqrt{\rho} u\Vert_{L^4}^2
     \Vert\nabla u\Vert_{L^2}\big).
\end{align}

\underline{step 3}.
Multiplying $\eqref{1}_2$ by $u_t$ and integrating it over $\Omega$, we deduce that
\begin{align}\label{11}
  &\frac{d}{dt}\int\mu(\theta)|D(u)|^2dx
   +\Vert\sqrt{\rho}u_t\Vert_{L^2}^2\nonumber\\
  &=\int\alpha\theta^{\alpha-1}\theta_t|D(u)|^2dx
   -\int\rho(u\cdot\nabla u)\cdot u_t dx\nonumber\\
  &\leq C\alpha\Vert\theta_t\Vert_{L^2}\Vert\nabla u\Vert_{L^4}^2
   +C\Vert\sqrt{\rho}u_t\Vert_{L^2}\Vert\sqrt{\rho} u\Vert_{L^4}\Vert\nabla u\Vert_{L^4}\nonumber\\
  &\leq C\alpha\Vert\theta_t\Vert_{L^2}\Vert\nabla u\Vert_{L^2}\Vert\nabla u\Vert_{H^1}
   +C\Vert\sqrt{\rho}u_t\Vert_{L^2}\Vert\sqrt{\rho} u\Vert_{L^4}
   \Vert\nabla u\Vert_{L^2}^{\frac{1}{2}}\Vert\nabla u\Vert_{H^1}^{\frac{1}{2}}\nonumber\\
  &\leq\frac{1}{2}\Vert\sqrt{\rho}u_t\Vert_{L^2}^2
   +C\alpha\Vert\theta_t\Vert_{L^2}^2\Vert\nabla u\Vert_{L^2}^2
   +C\alpha\Vert\theta_t\Vert_{L^2}\Vert\nabla u\Vert_{L^2}^2
    \Vert\sqrt{\rho} u\Vert_{L^4}^2
   +C\Vert\sqrt{\rho} u\Vert_{L^4}^4\Vert\nabla u\Vert_{L^2}^2\nonumber\\
  &\leq\frac{1}{2}\Vert\sqrt{\rho}u_t\Vert_{L^2}^2
   +C(\alpha\Vert\theta_t\Vert_{L^2}^2+\Vert\sqrt{\rho} u\Vert_{L^4}^4)
    \Vert\nabla u\Vert_{L^2}^2\nonumber\\
  &\leq\frac{1}{2}\Vert\sqrt{\rho}u_t\Vert_{L^2}^2
   +C(\alpha\Vert\theta_t\Vert_{L^2}^2+\Vert\nabla u\Vert_{L^2}^2)
    \Vert\nabla u\Vert_{L^2}^2\log(2+\Vert\nabla u\Vert_{L^2}^2)
\end{align}
where we have used \eqref{15}, \eqref{12} and Lemma \ref{L04}.
Then,  multiplying \eqref{11} by $1+t^2$, integrating the result with respect to $t$, we know that
\begin{align}\label{61}
  &(1+t^2)\Vert\nabla u\Vert_{L^2}^2
   +\int_0^t(1+s^2)\Vert\sqrt{\rho}u_t\Vert_{L^2}^2ds\nonumber\\
  &\leq C\Vert u_0\Vert_{H^1}^2
   +C\int_0^t(\alpha\Vert\theta_t\Vert_{L^2}^2+\Vert\nabla u\Vert_{L^2}^2)
    (1+s^2)\Vert\nabla u\Vert_{L^2}^2\log(2+\Vert\nabla u\Vert_{L^2}^2)ds.
\end{align}
Applying Bihari-LaSalle inequality (Lemma \ref{LA1}) with
\begin{align*}
  &y(t)=(1+t^2)\Vert\nabla u\Vert_{L^2}^2
   +\int_0^t(1+s^2)\Vert\sqrt{\rho}u_t\Vert_{L^2}^2ds,\\
  &h(t)=\alpha\Vert\theta_t\Vert_{L^2}^2+\Vert\nabla u\Vert_{L^2}^2,
  \ \ \ w(y)=y\log(2+y),
\end{align*}
one has
\begin{align}\label{52}
  (1+t^2)\Vert\nabla u\Vert_{L^2}^2
   +\int_0^t(1+s^2)\Vert\sqrt{\rho}u_t\Vert_{L^2}^2ds\leq C.
\end{align}
Therefore, assertion \eqref{9} follows now from \eqref{12},\eqref{14},\eqref{15} and \eqref{52}.
\end{proof}

\begin{lemma}\label{L3}
Under the assumptions of Proposition \ref{P0}, it holds that
\begin{gather}
  \sup_{0\leq t\leq T}\big((1+t^2)\Vert(\nabla^2u,\sqrt{\rho}u_t)\Vert_{L^2}^2\big)
   +\int_0^T(1+t^2)\Vert\nabla u_t\Vert_{L^2}^2dt
  \leq C, \label{80}\\
  \sup_{0\leq t\leq T}\Vert\rho\Vert_{W^{1,q}}
   \leq C\Vert\rho_0\Vert_{W^{1,q}}.
\end{gather}
\end{lemma}
\begin{proof}
Taking the operator $\partial_t$ to $\eqref{1}_2$, multiplying it by $u_t$, then integrating by parts over $\Omega$, we get that
\begin{align}
  &\frac{1}{2}\frac{d}{dt}\Vert\sqrt{\rho}u_t\Vert_{L^2}^2
  +2\int\mu(\theta)|D(u_t)|^2dx\nonumber\\
  &=\int\dive(\rho u)|u_t|^2dx
   +\int\dive(\rho u)u\cdot\nabla u\cdot u_tdx
   -\int\rho u_t\cdot\nabla u\cdot u_tdx
   -2\int\mu(\theta)_tD(u):\nabla u_tdx\nonumber\\
  &\triangleq\sum_{i=1}^{4}I_i.
\end{align}
We now estimate $I_i(i=1,2,3,4)$ as follows:
\begin{align*}
  I_1&=-2\int\rho u\cdot\nabla u_t\cdot u_tdx\\
  &\leq C\Vert u\Vert_{L^{\infty}}\Vert\nabla u_t\Vert_{L^2}\Vert\sqrt{\rho} u_t\Vert_{L^2}\\
  &\leq \frac{1}{4}\int\mu(\theta)|D(u_t)|^2dx
   +C\Vert u\Vert_{H^2}^2\Vert\sqrt{\rho} u_t\Vert_{L^2}^2,
   ~~~~~~~~~~~~~~~~~~~~~~~~~~~~~~~~~~~~~~~~~~~~~~~~
\end{align*}
\begin{align*}
  I_2&=-\int\rho u\cdot\nabla(u\cdot\nabla u\cdot u_t)dx\\
  &\leq C\int\big(\rho|u||\nabla u|^2|u_t|+\rho|u|^2|\nabla^2 u||u_t|
   +\rho|u|^2|\nabla u||\nabla u_t|\big)dx\\
  &\leq C\Vert u\Vert_{L^6}\Vert\nabla u\Vert_{L^2}\Vert\nabla u\Vert_{L^6}\Vert u_t\Vert_{L^6}
    +C\Vert u\Vert_{L^6}^2\Vert\nabla^2 u\Vert_{L^2}\Vert u_t\Vert_{L^6}
    +C\Vert u\Vert_{L^6}^2\Vert\nabla u\Vert_{L^6}\Vert\nabla u_t\Vert_{L^2}\\
  &\leq C\Vert\nabla u\Vert_{L^2}^2\Vert\nabla u\Vert_{H^1}\Vert\nabla u_t\Vert_{L^2}\\
  &\leq \frac{1}{4}\int\mu(\theta)|D(u_t)|^2dx
   +C\Vert\nabla u\Vert_{L^2}^4\big(\Vert\sqrt{\rho} u_t\Vert_{L^2}^2
   +1\big),
\end{align*}
\begin{align*}
  I_3&\leq C\Vert\sqrt{\rho}u_t\Vert_{L^4}^2\Vert\nabla u\Vert_{L^2}\\
  &\leq C\Vert\sqrt{\rho}u_t\Vert_{L^2}^{\frac{1}{2}}
   \Vert\sqrt{\rho}u_t\Vert_{L^6}^{\frac{3}{2}}\Vert\nabla u\Vert_{L^2}\\
  &\leq \frac{1}{4}\int\mu(\theta)|D(u_t)|^2dx
   +C\Vert\nabla u\Vert_{L^2}^4\Vert\sqrt{\rho} u_t\Vert_{L^2}^2,
   ~~~~~~~~~~~~~~~~~~~~~~~~~~~~~~~~~~~~~~~~~~~~~~~~~~~
\end{align*}
\begin{align*}
  I_4&\leq C\alpha\Vert\theta_t\Vert_{L^4}\Vert\nabla u\Vert_{L^4}\Vert\nabla u_t\Vert_{L^2} \\
  &\leq C\alpha\Vert\theta_t\Vert_{H^1}\Vert\nabla u\Vert_{H^1}\Vert\nabla u_t\Vert_{L^2} \\
  &\leq \frac{1}{4}\int\mu(\theta)|D(u_t)|^2dx
   +C\alpha\Vert\theta_t\Vert_{H^1}^2
   (\Vert\sqrt{\rho} u_t\Vert_{L^2}^2+\Vert\nabla u\Vert_{L^2}^2),
   ~~~~~~~~~~~~~~~~~~~~~~~~~~~~~~~~~~
\end{align*}
where one has used Gagliardo-Nirenberg inequality and \eqref{15}. Collecting all estimates of $I_i$, we infer that
\begin{align}\label{18}
  &\frac{d}{dt}\Vert\sqrt{\rho}u_t\Vert_{L^2}^2
   +2\int\mu(\theta)|D(u_t)|^2dx\nonumber\\
  &\leq C\big(\Vert u\Vert_{H^2}^2
   +\Vert\nabla u\Vert_{L^2}^4
   +\alpha\Vert\theta_t\Vert_{H^1}^2\big)
   \Vert\sqrt{\rho}u_t\Vert_{L^2}^2
   +C\Vert\nabla u\Vert_{L^2}^4
   +C\alpha\Vert\nabla u\Vert_{L^2}^2\Vert\theta_t\Vert_{H^1}^2.
\end{align}
We thus obtain after multiplying \eqref{18} by $1+t^2$ and using Gr\"{o}nwall's inequality that
\begin{align}\label{16}
 &(1+t^2)\Vert\sqrt{\rho}u_t\Vert_{L^2}^2
  +\int_0^t(1+s^2)\Vert\nabla u_t\Vert_{L^2}^2ds\nonumber\\
 &\leq C\big(\Vert\sqrt{\rho_0}u_{0t}\Vert_{L^2}^2
  +\int_0^t(1+s^2)(\Vert\nabla u\Vert_{H^1}^2
   +\alpha\Vert\nabla u\Vert_{L^2}^2\Vert\theta_t\Vert_{H^1}^2)ds
  +\int_0^ts\Vert\sqrt{\rho} u_t\Vert_{L^2}^2ds\big)\nonumber\\
 &\quad\cdot\exp\big(C\int_0^t(\Vert\nabla u\Vert_{H^1}^2+\Vert\nabla u\Vert_{L^2}^4
   +\alpha\Vert\theta_t\Vert_{H^1}^2)ds\big)\nonumber\\
 &\leq C,
\end{align}
owing to Lemma \ref{L2}, \eqref{60}, \eqref{17} and the simple fact that
\begin{align*}
  \Vert\sqrt{\rho_0}u_{0t}\Vert_{L^2}
  \leq C\Vert(\sqrt{\rho_0}u_0\cdot\nabla u_0, g_1)\Vert_{L^2}
  \leq C(\Vert u_0\Vert_{H^2}^2+1),
\end{align*}
which can be obtained by $\eqref{1}_2$ and the compatibility condition \eqref{62}. Hence, \eqref{80} is proved with the aid of \eqref{15} and \eqref{16}.

Finally, we estimate $\Vert\nabla\rho\Vert_{L^q}$. According to Lemma \ref{L03}, \eqref{60}, \eqref{17}, Gagliardo-Nirenberg inequality and Poincar\'{e}'s inequality, we know that
\begin{align}\label{24}
  \Vert\nabla u\Vert_{W^{1,r}}
  &\leq C\big(\Vert \rho u_t\Vert_{L^r}+\Vert\rho u\cdot\nabla u\Vert_{L^r}\big)
  \Big(1+\Vert\alpha\nabla\theta\Vert_{L^k}
    ^{\frac{k}{k-2}\cdot\frac{2r-2}{r}}\Big)\nonumber\\
  &\leq C\big(\Vert u_t\Vert_{L^r}+\Vert u\Vert_{L^{2r}}\Vert\nabla u\Vert_{L^{2r}}\big)\nonumber\\
  &\leq C\big(\Vert\nabla u_t\Vert_{L^2}
   +\Vert\nabla u\Vert_{L^2}
    \Vert\nabla u\Vert_{L^2}^{\frac{1}{r}}\Vert\nabla u\Vert_{H^1}^{\frac{r-1}{r}}\big)\nonumber\\
  &\leq C\big(\Vert\nabla u_t\Vert_{L^2}
   +\Vert\nabla u\Vert_{L^2}^2
   +\Vert\nabla u\Vert_{L^2}^{\frac{r+1}{r}}
    \Vert\nabla^2 u\Vert_{L^2}^{\frac{r-1}{r}}\big),
  ~~~\forall \, r\in(2,k), \ k\in(2,\infty),
\end{align}
which together with Lemma \ref{L2} and \eqref{80} implies
\begin{align*}
  &\int_0^T\Vert\nabla u\Vert_{L^{\infty}}dt
  \leq C\int_0^T\Vert\nabla u\Vert_{W^{1,r}}dt\\
  &\leq C\int_0^T\big(\Vert\nabla u_t\Vert_{L^2}
   +\Vert\nabla u\Vert_{L^2}^2
   +\Vert\nabla u\Vert_{L^2}^{\frac{r+1}{r}}
    \Vert\nabla^2 u\Vert_{L^2}^{\frac{r-1}{r}}\big)dt\\
  &\leq C\big(\int_0^Tt^2\Vert\nabla u_t\Vert_{L^2}^2dt\big)^{\frac{1}{2}}
   (\int_0^Tt^{-2}dt)^{\frac{1}{2}} +C\Vert\sqrt{\rho_0} u_0\Vert_{L^2}^2\\
  &\quad+C\sup_{0\leq t\leq T}(t\Vert\nabla^2 u\Vert_{L^2})^{\frac{r-1}{r}}
    \big(\int_0^T\Vert\nabla u\Vert_{L^2}^2dt\big)^{\frac{r+1}{2r}}
   (\int_0^Tt^{-\frac{r-1}{r}\frac{2r}{r-1}}dt)^{\frac{r-1}{2r}}\\
  &\leq C.
\end{align*}
We then deduce from $\eqref{1}_2$ that
\begin{align}\label{48}
  \frac{d}{dt}\Vert\nabla\rho\Vert_{L^q}\leq C\Vert\nabla u\Vert_{L^{\infty}}\Vert\nabla\rho\Vert_{L^q},
\end{align}
while it follows  that
\begin{align*}
  \Vert\nabla\rho\Vert_{L^q}\leq\Vert\nabla\rho_0\Vert_{L^q}
    \exp\big(C\int_0^T\Vert\nabla u\Vert_{L^{\infty}}dt\big)
  \leq C\Vert\nabla\rho_0\Vert_{L^q}.
\end{align*}
These complete the proof of Lemma \ref{L3}.
\end{proof}

We now get some estimates about $\theta$.

\begin{lemma}\label{L4}
Under the assumptions of Proposition \ref{P0}, it holds that
\begin{align*}
   \sup_{0\leq t\leq T}\big(\Vert\rho\theta^{\beta+2}\Vert_{L^1}
   +(1+t)\Vert\nabla\theta^{\beta+1}\Vert_{L^2}^2\big)
   +\int_0^T\big(\Vert\nabla\theta^{\beta+1}\Vert_{H^1}^2
   +(1+t)\Vert\sqrt{\rho}\theta^{\frac{\beta}{2}}\theta_t\Vert_{L^2}^2\big)dt
   \leq C.
\end{align*}
\end{lemma}
\begin{proof}

\underline{step 1}.
Multiply $\eqref{1}_3$ by $\theta^{\beta+1}$ and integrate over $\Omega$ yields
\begin{align*}
  &\frac{d}{dt}\int\rho\theta^{\beta+2}dx+\Vert\nabla\theta^{\beta+1}\Vert_{L^2}^2\\
  &\leq C\int\theta^{\alpha+\beta+1}|\nabla u|^2dx
   \leq C\Vert\theta^{\beta+1}\Vert_{L^2}\Vert\nabla u\Vert_{L^4}^2\\
  &\leq C\big(\Vert\rho\theta^{\beta+1}\Vert_{L^1}
   +\Vert\nabla\theta^{\beta+1}\Vert_{L^2}\big)
   \Vert\nabla u\Vert_{L^2}\Vert\nabla u\Vert_{H^1}\\
  &\leq \frac{1}{2}\Vert\nabla\theta^{\beta+1}\Vert_{L^2}^2
   +C\Vert\nabla u\Vert_{L^2}^2\Vert\nabla u\Vert_{H^1}^2
   +C\Vert\nabla u\Vert_{L^2}\Vert\nabla u\Vert_{H^1}\int\rho\theta^{\beta+2}dx,
\end{align*}
where we have used \eqref{17}, Lemma \ref{L08},  Gagliardo-Nirenberg inequality and Cauchy's inequality. We then obtain from Gr\"{o}nwall's inequality and Lemma \ref{L2} that
\begin{align}\label{19}
  &\Vert\rho\theta^{\beta+2}\Vert_{L^1}
   +\int_0^t\Vert\nabla\theta^{\beta+1}\Vert_{L^2}^2ds\nonumber\\
  &\leq C\big(\Vert\rho_0\theta_0^{\beta+2}\Vert_{L^1}
   +\int_0^t\Vert\nabla u\Vert_{L^2}^2\Vert\nabla u\Vert_{H^1}^2ds\big)
   \exp(C\int_0^t\Vert\nabla u\Vert_{H^1}^2ds)\nonumber\\
  &\leq C.
\end{align}

\underline{step 2}.
Taking $\eqref{1}_1$ and \eqref{3} into account, multiplying $\eqref{1}_3$ by $\kappa(\theta)\theta_t$, then integrating by parts over $\Omega$, we find that
\begin{align}\label{23}
  &\frac{1}{2}\frac{d}{dt}\Vert\kappa(\theta)\nabla\theta\Vert_{L^2}^2
   +\int\rho\kappa(\theta)|\theta_t|^2dx\nonumber\\
  &=-\int(\rho u\cdot\nabla\theta)\kappa(\theta)\theta_tdx
   +\int2\mu(\theta)|D(u)|^2\kappa(\theta)\theta_tdx\nonumber\\
  &\leq\frac{1}{4}\int\rho\kappa(\theta)|\theta_t|^2dx
   +C\Vert u\Vert_{L^{\infty}}^2\Vert\kappa(\theta)\nabla\theta\Vert_{L^2}^2
   +\frac{2}{\alpha+\beta+1}\frac{d}{dt}\int\theta^{\alpha+\beta+1}|D(u)|^2dx\nonumber\\
  &\quad+ C\Vert\theta^{\beta+1}\Vert_{L^4}\Vert\nabla u\Vert_{L^4}\Vert\nabla u_t\Vert_{L^2}\nonumber\\
  &\leq\frac{1}{4}\int\rho\kappa(\theta)|\theta_t|^2dx
   +\frac{2}{\alpha+\beta+1}\frac{d}{dt}\int\theta^{\alpha+\beta+1}|D(u)|^2dx\nonumber\\
  &\quad+C\Vert\nabla u\Vert_{H^1}^2(\Vert\nabla\theta^{\beta+1}\Vert_{L^2}^2
   +\Vert\rho\theta^{\beta+2}\Vert_{L^1}^2)
   +C\Vert\nabla u_t\Vert_{L^2}^2,
\end{align}
where one has used Lemma \ref{L08}.
Multiplying \eqref{23} by $1+t$ and integrating the resulting inequality over $[0,t]$, using \eqref{17}, Lemma \ref{L08}, \eqref{19} and Lemmas \ref{L2}-\ref{L3} leads to
\begin{align*}
  &(1+t)\Vert\nabla\theta^{\beta+1}\Vert_{L^2}^2
   +\int_0^t(1+s)\Vert\sqrt{\rho}\theta^{\frac{\beta}{2}}\theta_t\Vert_{L^2}^2ds\\
  &\leq C\Vert\nabla\theta_0^{\beta+1}\Vert_{L^2}^2
   +C(1+t)\Vert\theta^{\alpha+\beta+1}|\nabla u|^2\Vert_{L^1}
   +C\int_0^t\Vert\nabla\theta^{\beta+1}\Vert_{L^2}^2ds\\
  &\quad +C\int_0^t\Vert\nabla u\Vert_{H^1}^2(1+s)\Vert\nabla\theta^{\beta+1}\Vert_{L^2}^2ds
   +C\int_0^t(1+s)(\Vert\nabla u\Vert_{H^1}^2\Vert\rho\theta^{\beta+2}\Vert_{L^1}^2
   +\Vert\nabla u_t\Vert_{L^2}^2)ds\\
  &\leq C
   +C(1+t)(\Vert\rho\theta^{\beta+2}\Vert_{L^1}+\Vert\nabla\theta^{\beta+1}\Vert_{L^2})
    \Vert\nabla u\Vert_{H^1}^2
   +C\int_0^t\Vert\nabla u\Vert_{H^1}^2(1+s)\Vert\nabla\theta^{\beta+1}\Vert_{L^2}^2ds\\
  &\leq \frac{1}{2}(1+t)\Vert\nabla\theta^{\beta+1}\Vert_{L^2}^2
  +C +C\int_0^t\Vert\nabla u\Vert_{H^1}^2(1+s)\Vert\nabla\theta^{\beta+1}\Vert_{L^2}^2ds.
\end{align*}
This together with Gr\"{o}nwall's inequality and Lemma \ref{L2} gives
\begin{align}\label{20}
  (1+t)\Vert\nabla\theta^{\beta+1}\Vert_{L^2}^2
   +\int_0^t(1+s)\Vert\sqrt{\rho}\theta^{\frac{\beta}{2}}\theta_t\Vert_{L^2}^2ds
   \leq C.
\end{align}

\underline{step 3}. We now rewrite $\eqref{1}_3$ as
\begin{align}\label{69}
\left\{\begin{array}{l}
  -\Delta\theta^{\beta+1}
  =(\beta+1)\big(2\mu(\theta)|D(u)|^2-\rho\theta_t-\rho u\cdot\nabla\theta\big),
   ~~~~ \text{in} ~\Omega,\\
  \frac{\partial \theta}{\partial\mathbf{n}}=0,
  ~~~~~~~~~~~~~~~~~~~~~~~~~~~~~~~~~~~~~~~~~~~~~~ ~~~~~~~~~~~~~ \text{on} ~\partial\Omega.
\end{array}\right.
\end{align}
It gives rise to
\begin{align}\label{32}
  \Vert\Delta\theta^{\beta+1}\Vert_{L^2}
  &\leq C\Vert2\mu(\theta)|D(u)|^2-\rho\theta_t-\rho u\cdot\nabla\theta\Vert_{L^2}\nonumber\\
  &\leq C\Vert\nabla u\Vert_{L^4}^2
   +C\Vert\rho\theta_t\Vert_{L^2}
   +C\Vert u\Vert_{L^{\infty}}\Vert\nabla\theta^{\beta+1}\Vert_{L^2}\\
  &\leq C\Vert\nabla u\Vert_{H^1}^2+C\Vert\sqrt{\rho}\theta_t\Vert_{L^2}
   +C\Vert\nabla u\Vert_{H^1}\Vert\nabla\theta^{\beta+1}\Vert_{L^2},\nonumber
\end{align}
where we have used \eqref{17}.  It follows from  Lemma \ref{L06} and \eqref{32} that
\begin{align}\label{21}
  &\Vert\nabla^2\theta^{\beta+1}\Vert_{L^2}\leq C\big(\Vert\Delta\theta^{\beta+1}\Vert_{L^2}
   +\Vert\nabla\theta^{\beta+1}\Vert_{L^2}\big)\nonumber\\
  &\leq C\Vert\nabla u\Vert_{H^1}^2+C\Vert\sqrt{\rho}\theta_t\Vert_{L^2}
   +C(\Vert\nabla u\Vert_{H^1}+1)\Vert\nabla\theta^{\beta+1}\Vert_{L^2}.
\end{align}
We compute that
\begin{align}
  &\int_0^t\Vert\nabla^2\theta^{\beta+1}\Vert_{L^2}^2ds\nonumber\\
  &\leq C\int_0^t\big(\Vert\nabla u\Vert_{H^1}^4
   +\Vert\sqrt{\rho}\theta^{\frac{\beta}{2}}\theta_t\Vert_{L^2}^2
   +(\Vert\nabla u\Vert_{H^1}^2+1)\Vert\nabla\theta^{\beta+1}\Vert_{L^2}^2\big)ds
   \leq C,
\end{align}
owing to Lemmas \ref{L2}-\ref{L3}, \eqref{19} and \eqref{20}.
Moreover, in view of \eqref{3} and Lemma \ref{L05}, since
\begin{align}\label{58}
  \Vert\nabla^2\theta\Vert_{L^2}
  &\leq C\Vert\theta^{\beta}\nabla^2\theta\Vert_{L^2}
  =C\Vert\frac{1}{\beta+1}\nabla^2\theta^{\beta+1}
   -\beta\theta^{\beta-1}\nabla\theta\otimes\nabla\theta\Vert_{L^2}\nonumber\\
  &\leq C\Vert\nabla^2\theta^{\beta+1}\Vert_{L^2}
   +C\Vert\nabla\theta^{\beta+1}\Vert_{L^4}^2
   \leq C\Vert\nabla^2\theta^{\beta+1}\Vert_{L^2}
   (1+\Vert\nabla\theta^{\beta+1}\Vert_{L^2}),
\end{align}
we have
\begin{align}\label{22}
  \int_0^t\Vert\nabla^2\theta\Vert_{L^2}^2ds
  \leq C\int_0^t\Vert\nabla^2\theta^{\beta+1}\Vert_{L^2}^2
   (1+\Vert\nabla\theta^{\beta+1}\Vert_{L^2}^2)ds
  \leq C.
\end{align}
The proof of Lemma \ref{L4} is finished.
\end{proof}

\begin{lemma}\label{L5}
Under the assumptions of Proposition \ref{P0}, it holds that
\begin{align*}
  \sup_{0\leq t\leq T}\big((1+t^2)\Vert\sqrt{\rho}\theta_t\Vert_{L^2}^2\big)
   +\int_0^T(1+t^2)\Vert\theta^{\frac{\beta}{2}}\nabla\theta_t\Vert_{L^2}^2dt
  \leq  C\exp\big(\frac{C}{c_0}\big),
\end{align*}
where $c_0$ is a positive constant.  
\end{lemma}
\begin{proof}
Differentiating $\eqref{1}_3$ with respect to $t$, multiplying it by $\theta_t$, then integrating over $\Omega$ yields
\begin{align}\label{35}
  &\frac{1}{2}\frac{d}{dt}\int\rho|\theta_t|^2dx
   +\Vert\theta^{\frac{\beta}{2}}\nabla\theta_t\Vert_{L^2}^2\nonumber\\
  &=\int\dive(\rho u)|\theta_t|^2dx
   +\int\dive(\rho u)u\cdot\nabla\theta\theta_tdx
   -\int\rho u_t\cdot\nabla\theta\theta_tdx\nonumber\\
  &\quad +2\int\mu(\theta)(|D(u)|^2)_t\theta_tdx
   +2\int\mu(\theta)_t|D(u)|^2\theta_tdx
   -\int\kappa(\theta)_t\nabla\theta\cdot\nabla\theta_tdx\nonumber\\
  &\triangleq\sum_{i=5}^{10}I_i.
\end{align}
We estimate $I_i(i=5, \cdots, 10)$ in \eqref{35} term by term:
\begin{align*}
  I_5&=-2\int\rho u\cdot\nabla\theta_t\theta_tdx
  \leq C\Vert u\Vert_{L^{\infty}}\Vert\nabla\theta_t\Vert_{L^2}\Vert\sqrt{\rho}\theta_t\Vert_{L^2}\\
  &\leq \frac{1}{8}\Vert\theta^{\frac{\beta}{2}}\nabla\theta_t\Vert_{L^2}^2
    +C\Vert\nabla u\Vert_{H^1}^2\Vert\sqrt{\rho}\theta_t\Vert_{L^2}^2,
    ~~~~~~~~~~~~~~~~~~~~~~~~~~~~~~~~~~~~~~~~~~~~~~~~~~~~~~~~~~~~~~~~~~~
\end{align*}
\begin{align*}
  I_6&=-\int\rho u\cdot\nabla(u\cdot\nabla\theta\theta_t)dx
  \leq\int\rho|u|\big(|\nabla u||\nabla\theta||\theta_t|+|u||\nabla^2\theta||\theta_t|
    +|u||\nabla\theta||\nabla\theta_t|\big)dx\\
  &\leq C\Vert u\Vert_{L^{\infty}}\big(\Vert\nabla u\Vert_{L^4}\Vert\nabla\theta\Vert_{L^4}
    \Vert\sqrt{\rho}\theta_t\Vert_{L^2}+\Vert u\Vert_{L^{\infty}}\Vert\nabla^2\theta\Vert_{L^2}
    \Vert\sqrt{\rho}\theta_t\Vert_{L^2}+\Vert u\Vert_{L^{\infty}}\Vert\nabla\theta\Vert_{L^2}
    \Vert\nabla\theta_t\Vert_{L^2}\big)\\
  &\leq C\Vert\nabla u\Vert_{H^1}^2\Vert\nabla\theta\Vert_{H^1}\Vert\sqrt{\rho}\theta_t\Vert_{L^2}
    +C\Vert\nabla u\Vert_{H^1}^2\Vert\nabla\theta\Vert_{L^2}\Vert\nabla\theta_t\Vert_{L^2}\\
  &\leq\frac{1}{8}\Vert\theta^{\frac{\beta}{2}}\nabla\theta_t\Vert_{L^2}^2
    +C(\Vert\nabla u\Vert_{H^1}^2+\Vert\nabla u\Vert_{H^1}^4)\Vert\nabla\theta\Vert_{H^1}^2
    +C\Vert\nabla u\Vert_{H^1}^2\Vert\sqrt{\rho}\theta_t\Vert_{L^2}^2,
\end{align*}
\begin{align*}
  I_7&\leq C\Vert\rho u_t\Vert_{L^2}
   \Vert\nabla\theta\Vert_{L^4}\Vert\theta_t\Vert_{L^4}\\
  &\leq C\Vert\sqrt{\rho} u_t\Vert_{L^2}
   \Vert\nabla\theta\Vert_{H^1}\Vert\theta_t\Vert_{H^1}\\
  &\leq C\Vert\sqrt{\rho} u_t\Vert_{L^2}
   \Vert\nabla\theta\Vert_{H^1}
   (\Vert\sqrt{\rho}\theta_t\Vert_{L^2}+\Vert\nabla\theta_t\Vert_{L^2})\\
  &\leq \frac{1}{8}\Vert\theta^{\frac{\beta}{2}}\nabla\theta_t\Vert_{L^2}^2
   +C\Vert\sqrt{\rho} u_t\Vert_{L^2}^2(\Vert\nabla\theta\Vert_{H^1}^2+1)
   +C\Vert\nabla\theta\Vert_{H^1}^2\Vert\sqrt{\rho}\theta_t\Vert_{L^2}^2,
   ~~~~~~~~~~~~~~~~~~~~~~~~~~~~~~~
\end{align*}
\begin{align*}
  I_8&\leq C\Vert\nabla u\Vert_{L^4}
   \Vert\nabla u_t\Vert_{L^2}\Vert\theta_t\Vert_{L^4}\\
  &\leq C\Vert\nabla u\Vert_{H^1}\Vert\nabla u_t\Vert_{L^2}\Vert\theta_t\Vert_{H^1}\\
  &\leq C\Vert\nabla u\Vert_{H^1}\Vert\nabla u_t\Vert_{L^2}
  (\Vert\sqrt{\rho}\theta_t\Vert_{L^2}+\Vert\nabla\theta_t\Vert_{L^2})\\
  &\leq \frac{1}{8}\Vert\theta^{\frac{\beta}{2}}\nabla\theta_t\Vert_{L^2}^2
   +C(\Vert\nabla u\Vert_{H^1}^2+1)\Vert\nabla u_t\Vert_{L^2}^2
   +C\Vert\nabla u\Vert_{H^1}^2\Vert\sqrt{\rho}\theta_t\Vert_{L^2}^2,
   ~~~~~~~~~~~~~~~~~~~~~~~~~~~~~~~~
\end{align*}
\begin{align*}
  I_9&=2\alpha\int\theta^{\alpha-1}|D(u)|^2|\theta_t|^2dx
  \leq C\alpha\Vert\nabla u\Vert_{L^4}^2\Vert\theta_t\Vert_{L^4}^2\\
  &\leq C\alpha\Vert\nabla u\Vert_{H^1}^2
   (\Vert\sqrt{\rho}\theta_t\Vert_{L^2}^2+\Vert\nabla\theta_t\Vert_{L^2}^2)\\
  &\leq C\alpha\Vert\theta^{\frac{\beta}{2}}\nabla\theta_t\Vert_{L^2}^2
   +C\Vert\nabla u\Vert_{H^1}^2\Vert\sqrt{\rho}\theta_t\Vert_{L^2}^2,
   ~~~~~~~~~~~~~~~~~~~~~~~~~~~~~~~~~~~~~~~~~~~~~~~~~~~~~~~~~~~~~~
\end{align*}
\begin{align*}
  I_{10}&=-\int\beta\theta^{\beta-1}\theta_t\nabla\theta\cdot\nabla\theta_tdx
  \leq C\Vert\theta_t\Vert_{L^4}\Vert\nabla\theta^{\beta+1}\Vert_{L^4}
   \Vert\nabla\theta_t\Vert_{L^2}\\
  &\leq C\Vert\theta_t\Vert_{L^2}^{\frac{1}{2}}\Vert\theta_t\Vert_{H^1}^{\frac{1}{2}}
   \Vert\nabla\theta^{\beta+1}\Vert_{L^2}^{\frac{1}{2}}
   \Vert\nabla^2\theta^{\beta+1}\Vert_{L^2}^{\frac{1}{2}}\Vert\nabla\theta_t\Vert_{L^2}\\
  &\leq \frac{1}{8}\Vert\theta^{\frac{\beta}{2}}\nabla\theta_t\Vert_{L^2}^2
   +C\Vert\theta_t\Vert_{L^2}^2
   \big(\Vert\nabla\theta^{\beta+1}\Vert_{L^2}
   \Vert\nabla^2\theta^{\beta+1}\Vert_{L^2}
   +\Vert\nabla\theta^{\beta+1}\Vert_{L^2}^2
   \Vert\nabla^2\theta^{\beta+1}\Vert_{L^2}^2\big),
  ~~~~~~~~~~~~~
\end{align*}
where we have used Gagliardo-Nirenberg's inequality, \eqref{17}, Lemma \ref{L08}, Lemma \ref{L2} and Lemma \ref{L3}.

In view of Lemma \ref{L08},
\begin{align}\label{82}
  \Vert\theta_t\Vert_{L^2}^2
  &=\int_{\{x\in\Omega|\rho(x,t)\leq c_0\}}\theta_t^2dx
  +\int_{\{x\in\Omega|\rho(x,t)> c_0\}}\theta_t^2dx\nonumber\\
  &\leq \big(\int_{\{x\in\Omega|\rho(x,t)\leq c_0\}} dx\big)^{\frac{1}{2}}
   \big(\int_{\{x\in\Omega|\rho(x,t)\leq c_0\}}\theta_t^4dx\big)^{\frac{1}{2}}
   +\frac{1}{c_0}\Vert\sqrt{\rho}\theta_t\Vert_{L^2}^2\nonumber\\
  &\leq \big|\{x\in\Omega|\rho_0(x)\leq c_0\}\big|^{\frac{1}{2}}
   C(\Vert\sqrt{\rho}\theta_t\Vert_{L^2}^2+\Vert\nabla\theta_t\Vert_{L^2}^2)
   +\frac{1}{c_0}\Vert\sqrt{\rho}\theta_t\Vert_{L^2}^2\nonumber\\
  &\leq C|V|^{\frac{1}{2}}
   (\Vert\sqrt{\rho}\theta_t\Vert_{L^2}^2+\Vert\nabla\theta_t\Vert_{L^2}^2)
   +\frac{1}{c_0}\Vert\sqrt{\rho}\theta_t\Vert_{L^2}^2,
\end{align}
owing to $\big|\{x\in\Omega|\rho(x,t)\leq c_0\}\big|=\big|\{x\in\Omega|\rho_0(x)\leq c_0\}\big|$ (see \cite[Theorem 2.1]{Lions}). Moreover, from \eqref{60} and Lemma \ref{L4}, we have
\begin{align*}
  \Vert\nabla^2\theta^{\beta+1}\Vert_{L^2}
  &\leq \Vert(\beta+1)\theta^{\beta}\nabla^2\theta
   +\beta(\beta+1)\theta^{\beta-1}\nabla\theta\otimes\nabla\theta\Vert_{L^2}\\
  &\leq C\Vert\theta\Vert_{L^{\infty}}^{\beta}\Vert\nabla^2\theta\Vert_{L^2}
   +C\Vert\nabla\theta^{\beta+1}\Vert_{L^4}\Vert\nabla\theta\Vert_{L^4}\\
  &\leq C\Vert\theta\Vert_{H^2}^{\beta}\Vert\nabla^2\theta\Vert_{L^2}
   +C\Vert\nabla\theta^{\beta+1}\Vert_{L^2}^{\frac{1}{2}}
   \Vert\nabla^2\theta^{\beta+1}\Vert_{L^2}^{\frac{1}{2}}
   \Vert\nabla\theta\Vert_{L^2}^{\frac{1}{2}}
   \Vert\nabla^2\theta\Vert_{L^2}^{\frac{1}{2}}\\
  &\leq \frac{1}{2}\Vert\nabla^2\theta^{\beta+1}\Vert_{L^2}
   +C\Vert\nabla^2\theta\Vert_{L^2}+C\Vert\nabla^2\theta\Vert_{L^2}^{\beta+1}\\
  &\leq \frac{1}{2}\Vert\nabla^2\theta^{\beta+1}\Vert_{L^2}
   +CM^{\frac{\beta+1}{2}}.
\end{align*}
Thus,
\begin{align*}
  I_{10}\leq \frac{1}{8}\Vert\theta^{\frac{\beta}{2}}\nabla\theta_t\Vert_{L^2}^2
   +C|V|^{\frac{1}{2}}M^{\beta+1}\Vert\theta^{\frac{\beta}{2}}\nabla\theta_t\Vert_{L^2}^2
   +\frac{C}{c_0}\big(\Vert\nabla\theta^{\beta+1}\Vert_{L^2}^2
   +\Vert\nabla^2\theta^{\beta+1}\Vert_{L^2}^2\big)
   \Vert\sqrt{\rho}\theta_t\Vert_{L^2}^2.
\end{align*}
Substituting the estimates of $I_i(i=5, \cdots, 10)$ into \eqref{35},  letting
\begin{align}\label{63}
 C\alpha\leq\frac{1}{8}, ~~~~~ C|V|^{\frac{1}{2}}M^{\beta+1}\leq\frac{1}{8},
\end{align}
we obtain
\begin{align}\label{25}
  \frac{d}{dt}\big((1+t^2)\Vert\sqrt{\rho}\theta_t\Vert_{L^2}^2\big)
   +(1+t^2)\Vert\theta^{\frac{\beta}{2}}\nabla\theta_t\Vert_{L^2}^2
  \leq C\mathcal{A}(t)(1+t^2)\Vert\sqrt{\rho}\theta_t\Vert_{L^2}^2
   +C\mathcal{B}(t),
\end{align}
where
\begin{align}\label{37}
  \left\{
    \begin{array}{ll}
        \mathcal{A}(t)
  \triangleq\Vert\nabla u\Vert_{H^1}^2
   +\Vert\nabla\theta\Vert_{H^1}^2
   +\frac{1}{c_0}\Vert\nabla\theta^{\beta+1}\Vert_{H^1}^2,\\[2mm]
  \mathcal{B}(t)
  \triangleq(1+t^2)\big((\Vert\nabla u\Vert_{H^1}^2+\Vert\sqrt{\rho}u_t\Vert_{L^2}^2)
  \Vert\nabla\theta\Vert_{H^1}^2
  +\Vert(\sqrt{\rho}u_t,\nabla u_t)\Vert_{L^2}^2\big)
  +t\Vert\sqrt{\rho}\theta_t\Vert_{L^2}^2.
    \end{array}
  \right.
\end{align}
It follows from Lemma \ref{L05}, Lemmas \ref{L2}-\ref{L4}  and \eqref{22} that
\begin{align*}
  \int_0^t\mathcal{A}(s)ds
  &\leq C\int_0^t\big(\Vert\nabla u\Vert_{H^1}^2
   +\Vert\nabla^2\theta\Vert_{L^2}^2
   +\frac{1}{c_0}\Vert\nabla^2\theta^{\beta+1}\Vert_{L^2}^2\big)ds
  \leq \frac{C}{c_0},
\end{align*}
\begin{align*}
  \int_0^t\mathcal{B}(s)ds
  &\leq C\sup_{0\leq t\leq T}\big((1+t^2)(\Vert\nabla u\Vert_{H^1}^2
    +\Vert\sqrt{\rho}u_t\Vert_{L^2}^2)\big)
   \int_0^t\Vert\nabla\theta\Vert_{H^1}^2ds\\
  &\quad+C\int_0^t\big((1+s^2)\Vert(\sqrt{\rho}u_t,\nabla u_t)\Vert_{L^2}^2
   +s\Vert\sqrt{\rho}\theta_t\Vert_{L^2}^2\big)ds\\
  &\leq C.
\end{align*}
Next, \eqref{25}, Gr\"{o}nwall's inequality and straight calculations show that
\begin{align}\label{26}
  &(1+t^2)\Vert\sqrt{\rho}\theta_t\Vert_{L^2}^2
   +\int_0^t(1+s^2)\Vert\theta^{\frac{\beta}{2}}\nabla\theta_t\Vert_{L^2}^2ds\nonumber\\
  &\leq C\big(\Vert\sqrt{\rho_0}\theta_{0t}\Vert_{L^2}^2+\int_0^t\mathcal{B}(s)ds\big)
   \exp\big(C\int_0^t\mathcal{A}(s)ds\big)\nonumber\\
  &\leq C\exp\big(\frac{C}{c_0}\big)
\end{align}
owing to the simple fact that
\begin{align*}
  \Vert\sqrt{\rho_0}\theta_{0t}\Vert_{L^2}
  \leq C\Vert(\sqrt{\rho_0}u_0\cdot\nabla\theta_0, g_2)\Vert_{L^2}
  \leq C(\Vert u_0\Vert_{H^1}\Vert\nabla\theta_0\Vert_{H^1}+\Vert g_2\Vert_{L^2})
  \leq C.
\end{align*}
Assertion follows now from \eqref{26}.
\end{proof}

\vspace{0.2cm}

\noindent
\textbf{Proof of Proposition \ref{P0}:}

First, Lemma \ref{L08}, Lemmas \ref{L2}-\ref{L5}, \eqref{21}\eqref{58} and standard calculations give that
\begin{align}\label{27}
  &\sup_{0\leq t\leq T}\Vert\theta\Vert_{H^2}^2
   +\int_0^T\big((1+t)\Vert\theta_t\Vert_{L^2}^2
   +(1+t^2)\Vert\nabla\theta_t\Vert_{L^2}^2\big)dt\nonumber\\
  &\leq C\sup_{0\leq t\leq T}(\Vert\rho\theta^{\beta+2}\Vert_{L^1}+\Vert\nabla\theta\Vert_{H^1}^2)
   +\int_0^T\big((1+t)\Vert\sqrt{\rho}\theta_t\Vert_{L^2}^2
   +(1+t^2)\Vert\nabla\theta_t\Vert_{L^2}^2\big)dt\nonumber\\
  &\leq C\exp\big(\frac{C}{c_0}\big),
\end{align}
provided
\begin{align}
\alpha\leq\min\{1, M^{-1}, (8C)^{-1}\}, ~~~~~~
C|V|^{\frac{1}{2}}M^{\beta+1}\leq\frac{1}{8}.
\end{align}
Then, choosing $c_0, M$ and $\epsilon_0$ such that
\begin{align}\label{28}
\left\{
  \begin{array}{ll}
    c_0=\min\{(2C\beta+2C+2)^{-1}, (\log(8C^{\beta+2}))^{-1}\}, \\[2mm]
   M=C\exp\big(\frac{C}{c_0}\big), \\[2mm]
   \epsilon_0=\min\{1, M^{-1}, (8C)^{-1}\}.
  \end{array}
\right.
\end{align}
according to $\eqref{79}$, we have
\begin{align}\label{64}
 C|V|^{\frac{1}{2}}M^{\beta+1}
 &\leq C^{\beta+2}\exp\big(\frac{1}{c_0}(-\frac{1}{2c_0}+C\beta+C)\big)\nonumber\\
 &\leq C^{\beta+2}\exp(-\frac{1}{c_0})\nonumber\\
 &\leq C^{\beta+2}\exp(-\log(8C^{\beta+2}))\nonumber\\
 &=\frac{1}{8}.
\end{align}
Thus, we directly obtain \eqref{65} from \eqref{27}-\eqref{64}. Therefore, Proposition \ref{P0} is valid.
\hfill $\qedsymbol$

In the end of this subsection, we introduce a higher-order estimate lemma, which is used to extend local strong solution to global ones.

\begin{lemma}\label{L8}
Under the assumptions of Proposition \ref{P0}, it holds that
\begin{align}\label{74}
  \sup_{0\leq t\leq T}\Vert\rho_t\Vert_{L^q}
   +\int_0^T\Vert(\nabla u, \frac{P}{\mu(\theta)})\Vert_{H^2}^2dt
  \leq  C_0, ~~~~~~~~
  \int_0^T\Vert\nabla^3\theta\Vert_{L^2}^2dt
  \leq  C_0(T),
\end{align}
where $C_0(T)$ is a positive constant depending on the initial data and the time $T$.
\end{lemma}
\begin{proof}
Before the proof begins, let us introduce the symbol ``$\lesssim$''. If $A\lesssim B$, it means that there exists a positive constant $C_0$ such that $A\leq C_0B$.

\underline{step 1}.  It follows from $\eqref{1}_1$ and Lemmas \ref{L2}-\ref{L3} that
\begin{align}\label{72}
  \Vert\rho_t\Vert_{L^q}
  =\Vert u\cdot\nabla\rho\Vert_{L^q}
  \leq\Vert u\Vert_{L^{\infty}}\Vert\nabla\rho\Vert_{L^q}
  \leq C\Vert u\Vert_{H^2}\Vert\nabla\rho\Vert_{L^q}
  \leq C_0.
\end{align}

\underline{step 2}. Applying the standard $H^3$-estimates (Lemma \ref{L03}) to the Stokes equations \eqref{70} yields that
\begin{align}\label{71}
  &\Vert u\Vert_{H^3}+\Vert\frac{P}{\mu(\theta)}\Vert_{H^2}\nonumber\\
  &\lesssim \Vert\rho u_t+\rho u\cdot\nabla u\Vert_{H^1}\nonumber\\
  &\lesssim \Vert\sqrt{\rho} u_t\Vert_{L^2}+\Vert u\Vert_{L^4}\Vert\nabla u\Vert_{L^4}
   +\Vert\nabla\rho\Vert_{L^q}\Vert u_t\Vert_{L^{\frac{2q}{q-2}}}
   +\Vert\nabla u_t\Vert_{L^2}\nonumber\\
  &\quad +\Vert\nabla\rho\Vert_{L^q}\Vert u\Vert_{L^{\infty}}\Vert\nabla u\Vert_{L^{\frac{2q}{q-2}}}
   +\Vert\nabla u\Vert_{L^4}^2+\Vert u\Vert_{L^{\infty}}\Vert\nabla^2 u\Vert_{L^2}\nonumber\\
  &\lesssim \Vert\sqrt{\rho} u_t\Vert_{L^2}
   +\Vert\nabla u\Vert_{H^1}+\Vert\nabla u_t\Vert_{L^2},
\end{align}
where we have used Lemma \ref{L08} and Lemmas \ref{L1}-\ref{L5}.

\underline{step 3}. Next, we rewrite \eqref{69} as
\begin{align*}
\left\{\begin{array}{l}
  -\Delta\theta=\beta\theta^{-1}\nabla\theta\cdot\nabla\theta
  +\theta^{-\beta}(2\mu(\theta)|D(u)|^2-\rho\theta_t-\rho u\cdot\nabla\theta),
   ~~~~ \text{in} ~\Omega,\\
  \frac{\partial \theta}{\partial\mathbf{n}}=0,
  ~~~~~~~~~~~~~~~~~~~~~~~~~~~~~~~~~~~~~~~~~~~~~~~~~~~~~~~~~ \, ~~~~~~~~~~~~~ \text{on} ~\partial\Omega.
\end{array}\right.
\end{align*}
Hence, the standard $H^3$-estimate for the Neumann problem to the elliptic equation\cite{Lieberman,Zhong2020} gives rise to
\begin{align}\label{73}
  \Vert\theta\Vert_{H^3}
  &\lesssim\Vert\beta\theta^{-1}\nabla\theta\cdot\nabla\theta
  +\theta^{-\beta}(2\mu(\theta)|D(u)|^2-\rho\theta_t-\rho u\cdot\nabla\theta)\Vert_{H^1}
  +\Vert\theta\Vert_{H^1}\nonumber\\
  &\lesssim\Vert\nabla\theta\Vert_{L^4}^2
  +\Vert\nabla u\Vert_{L^4}^2+\Vert\sqrt{\rho}\theta_t\Vert_{L^2}
  +\Vert u\Vert_{L^{\infty}}\Vert\nabla\theta\Vert_{L^2}
  +\Vert\nabla\theta\Vert_{L^6}^3+\Vert\nabla^2\theta\Vert_{L^4}\Vert\nabla\theta\Vert_{L^4}\nonumber\\
  &\quad+\Vert\nabla\theta\Vert_{L^6}\Vert\nabla u\Vert_{L^6}^2
  +\Vert\nabla^2u\Vert_{L^4}\Vert\nabla u\Vert_{L^4}
  +\Vert\nabla\theta\Vert_{L^4}\Vert\theta_t\Vert_{L^4}
  +\Vert\nabla\rho\Vert_{L^q}\Vert\theta_t\Vert_{L^{\frac{2q}{q-2}}}
  +\Vert\nabla\theta_t\Vert_{L^2}\nonumber\\
  &\quad+\Vert u\Vert_{L^{\infty}}\Vert\nabla\theta\Vert_{L^4}^2
  +\Vert\nabla\rho\Vert_{L^q}\Vert u\Vert_{L^{\infty}}\Vert\nabla\theta\Vert_{L^{\frac{2q}{q-2}}}
  +\Vert u\Vert_{L^{\infty}}\Vert\nabla^2\theta\Vert_{L^2}+\Vert\theta\Vert_{H^1}\nonumber\\
  &\lesssim\Vert\nabla\theta\Vert_{H^1}^2
  +\Vert\nabla u\Vert_{H^1}^2+\Vert\sqrt{\rho}\theta_t\Vert_{L^2}+\Vert\nabla\theta\Vert_{L^2}
  +(\Vert\nabla^2\theta\Vert_{L^2}^{\frac{1}{2}}\Vert\nabla^3\theta\Vert_{L^2}^{\frac{1}{2}}
   +\Vert\nabla^2\theta\Vert_{L^2})\Vert\nabla\theta\Vert_{H^1}\nonumber\\
  &\quad +(\Vert\nabla^2u\Vert_{L^2}^{\frac{1}{2}}\Vert\nabla^3u\Vert_{L^2}^{\frac{1}{2}}
   +\Vert\nabla^2u\Vert_{L^2})\Vert\nabla u\Vert_{H^1}
  +\Vert\nabla\theta_t\Vert_{L^2}+\Vert\nabla\theta\Vert_{H^1}+\Vert\theta\Vert_{H^1}\nonumber\\
  &\lesssim\frac{1}{2}\Vert\nabla^3\theta\Vert_{L^2}+\Vert\nabla\theta\Vert_{H^1}
  +\Vert\nabla u\Vert_{H^1}+\Vert\sqrt{\rho}\theta_t\Vert_{L^2}+\Vert\nabla^3u\Vert_{L^2}
  +\Vert\nabla\theta_t\Vert_{L^2}+\Vert\rho\theta^{\beta+2}\Vert_{L^1}\nonumber\\
  &\lesssim\frac{1}{2}\Vert\nabla^3\theta\Vert_{L^2}
  +\Vert(\nabla u,\nabla\theta)\Vert_{H^1}
  +\Vert(\sqrt{\rho}u_t,\sqrt{\rho}\theta_t,\nabla u_t,\nabla\theta_t)\Vert_{L^2}+C,
\end{align}
where one has used Lemma \ref{L08}, Lemmas \ref{L1}-\ref{L5} and \eqref{71}.

Combining \eqref{72}, \eqref{71} with \eqref{73} gives \eqref{74} and finishes the proof of Lemma \ref{L8}.
\end{proof}

\subsection{Time-Weighted Estimates}

In this subsection,  our main effort is focused on the time-weighted estimates of the solution, with the aim of obtaining the decay properties of $(u,\theta)$.

Firstly, collecting Lemmas \ref{L1}-\ref{L8}, we obtain

\begin{Proposition}\label{P1}
Under the assumptions of Proposition \ref{P0},  $\forall \, (x,t)\in\Omega\times[0,T]$, it holds that
\begin{align}
  &\qquad\qquad\qquad 0\leq\rho(x,t)\leq\tilde{\rho}, ~~~~~~~~~~
  \theta(x,t)\geq \underline{\theta},\label{66}\\
  &\sup_{0\leq t\leq T}\big(\Vert\rho\Vert_{W^{1,q}}+\Vert\rho_t\Vert_{L^q}
  +\Vert(u, \theta)\Vert_{H^2}+\Vert(\sqrt{\rho}u_t, \sqrt{\rho}\theta_t)\Vert_{L^2}\big)\nonumber\\
  &\quad+\int_0^T\Vert(\nabla u,\frac{P}{\mu(\theta)})\Vert_{H^2}^2+
  \Vert(u_t, \nabla\theta, \theta_t)\Vert_{H^1}^2dt\leq C_0,\\
  &\qquad\qquad\qquad\quad\int_0^T\Vert\nabla^3\theta\Vert_{L^2}^2dt\leq C_0(T).
\end{align}
\end{Proposition}

We then derive the following decay estimates on the  velocity.
\begin{lemma}\label{L6}
Under the assumptions of Proposition \ref{P0}, it holds that
\begin{align}\label{42}
  \sup_{0\leq t\leq T}\big(e^{\sigma_1 t}(\Vert u\Vert_{H^2}^2+\Vert\sqrt{\rho} u_t\Vert_{L^2}^2
   +\Vert P\Vert_{H^1}^2)\big)
  +\int_0^Te^{\sigma_1 t}\Vert(\nabla u, u_t)\Vert_{H^1}^2dt\leq C_0,
\end{align}
where $\displaystyle\sigma_1\triangleq\frac{\pi^2\underline{\theta}^{\alpha}}{\tilde{\rho}d^2}$ and $d=\text{diam}(\Omega)\triangleq\sup\{|x-y| | x,y\in\Omega\}$.
\end{lemma}
\begin{proof}
\underline{step 1}. Thanks to Poincar\'{e}'s inequality (see\cite{Galdi}, p.70) and \eqref{66}, one obtains
\begin{align*}
  \Vert\sqrt{\rho}u\Vert_{L^2}^2\leq\tilde{\rho}\frac{d^2}{\pi^2}\Vert\nabla u\Vert_{L^2}^2
  \leq\frac{2\tilde{\rho}d^2}{\pi^2\underline{\theta}^{\alpha}}
   \Vert\theta^{\frac{\alpha}{2}}D(u)\Vert_{L^2}^2.
\end{align*}
Multiplying \eqref{10} by $e^{\sigma_1 t}$, we find
\begin{align*}
  \frac{d}{dt}\big(e^{\sigma_1 t}\Vert\sqrt{\rho}u\Vert_{L^2}^2\big)
  +2e^{\sigma_1 t} \Vert\theta^{\frac{\alpha}{2}}D(u)\Vert_{L^2}^2\leq0,
\end{align*}
which means that
\begin{equation}\label{33}
  e^{\sigma_1 t}\Vert\sqrt{\rho} u\Vert_{L^2}^2
  +\int_0^t e^{\sigma_1 s}\Vert\nabla u\Vert_{L^2}^2ds\leq C_0.
\end{equation}

\underline{step 2}. Since \eqref{11} and Lemma \ref{L08}, we can check that
\begin{align}\label{43}
  &\frac{d}{dt}\int\mu(\theta)|D(u)|^2dx
   +\frac{1}{2}\Vert\sqrt{\rho}u_t\Vert_{L^2}^2\nonumber\\
  &\leq C\big(\Vert\theta_t\Vert_{L^2}^2+\Vert u\Vert_{L^4}^4\big)
   \Vert\nabla u\Vert_{L^2}^2\nonumber\\
  &\leq C\big(\Vert\sqrt{\rho}\theta_t\Vert_{L^2}^2+\Vert\nabla\theta_t\Vert_{L^2}^2
   +\Vert\nabla u\Vert_{L^2}^4\big)
   \Vert\nabla u\Vert_{L^2}^2.
\end{align}
Multiplying \eqref{43} by $e^{\sigma_1 t}$ yields
\begin{align*}
  &\frac{d}{dt}\big(e^{\sigma_1 t}\Vert\theta^{\frac{\alpha}{2}}D(u)\Vert_{L^2}^2\big)
   +\frac{1}{2}e^{\sigma_1 t}\Vert\sqrt{\rho}u_t\Vert_{L^2}^2\\
  &\leq C\big(\Vert\sqrt{\rho}\theta_t\Vert_{L^2}^2+\Vert\nabla\theta_t\Vert_{L^2}^2
   +\Vert\nabla u\Vert_{L^2}^4\big)
   e^{\sigma_1 t}\Vert\theta^{\frac{\alpha}{2}}D(u)\Vert_{L^2}^2
   +\sigma_1 e^{\sigma_1 t}\Vert\theta^{\frac{\alpha}{2}}D(u)\Vert_{L^2}^2.
\end{align*}
It follows from Gr\"{o}nwall's inequality, Proposition \ref{P1} and \eqref{33} that
\begin{align}\label{34}
  e^{\sigma_1 t}\Vert\nabla u\Vert_{L^2}^2
   +\int_0^te^{\sigma_1 s}\Vert\sqrt{\rho} u_t\Vert_{L^2}^2ds\leq C_0.
\end{align}

\underline{step 3}.
 we deduce from \eqref{18} and Lemma \ref{L2} that
\begin{align*}
  &\frac{d}{dt}\big(e^{\sigma_1 t}\Vert\sqrt{\rho}u_t\Vert_{L^2}^2\big)
  +e^{\sigma_1 t}\Vert\theta^{\frac{\alpha}{2}}D(u_t)\Vert_{L^2}^2\\
  &\leq C\big(\Vert u\Vert_{H^2}^2+\Vert\nabla u\Vert_{L^2}^4
    +\Vert\theta_t\Vert_{H^1}^2\big)
   e^{\sigma_1 t}\Vert\sqrt{\rho}u_t\Vert_{L^2}^2\\
  &\quad +Ce^{\sigma_1 t}\big(\Vert\sqrt{\rho}u_t\Vert_{L^2}^2
   +\Vert\nabla u\Vert_{L^2}^4
   +\Vert\nabla u\Vert_{L^2}^2\Vert\theta_t\Vert_{H^1}^2\big)\\
  &\leq C\big(\Vert\nabla u\Vert_{H^1}^2
    +\Vert(\sqrt{\rho}\theta_t,\nabla\theta_t)\Vert_{L^2}^2\big)
   e^{\sigma_1 t}\Vert\sqrt{\rho}u_t\Vert_{L^2}^2\\
  &\quad +Ce^{\sigma_1 t}\big(\Vert\sqrt{\rho}u_t\Vert_{L^2}^2
   +\Vert\nabla u\Vert_{L^2}^2\big)
   +C\sup_{0\leq t\leq T}(e^{\sigma_1 t}\Vert\nabla u\Vert_{L^2}^2)
   \Vert(\sqrt{\rho}\theta_t,\nabla\theta_t)\Vert_{L^2}^2.
\end{align*}
This together with Gr\"{o}nwall's inequality, \eqref{33}, \eqref{34} and Proposition \ref{P1} implies
\begin{align}\label{36}
  e^{\sigma_1 t}\Vert\sqrt{\rho} u_t\Vert_{L^2}^2
  +\int_0^te^{\sigma_1 s}\Vert\nabla u_t\Vert_{L^2}^2ds\leq C_0.
\end{align}
Therefore, we arrive at \eqref{42} from \eqref{33}-\eqref{36} and \eqref{15} immediately.
\end{proof}

Now we derive the following decay estimates on the temperature.
\begin{lemma}\label{L7}
Under the assumptions of Proposition \ref{P0}, it holds that
\begin{align}\label{40}
  \sup_{0\leq t\leq T}\big(e^{\sigma_2 t}
   (\Vert\theta-\frac{1}{\overline{\rho_0}|\Omega|}E_0\Vert_{H^2}^2
   +\Vert\sqrt{\rho}\theta_t\Vert_{L^2}^2)\big)
  +\int_0^Te^{\sigma_2 t}\Vert(\nabla\theta, \theta_t)\Vert_{H^1}^2dt\leq C_0,
\end{align}
where $\displaystyle \sigma_2\triangleq\frac{\pi^2}{\tilde{\rho}d^2}
\min\big\{\frac{1}{2}\underline{\theta}^{\beta}(1+\frac{\tilde{\rho}}{\overline{\rho_0}})^{-2}, \underline{\theta}^{\alpha}\big\}, d=\text{diam}(\Omega)\triangleq\sup\{|x-y| | x,y\in\Omega\}$,
and
$\displaystyle E_0=\int\rho_0(\theta_0+\frac{1}{2}|u_0|^2)dx$,
$\displaystyle \overline{\rho_0}=\frac{1}{|\Omega|}\int\rho_0dx$ are two positive constants.
\end{lemma}

\begin{proof}

\underline{step 1}.
Direct calculations together with \eqref{1} lead to
\begin{align*}
  \int(\rho\theta+\frac{1}{2}\rho|u|^2)dx
  =\int(\rho_0\theta_0+\frac{1}{2}\rho_0|u_0|^2)dx= E_0.
\end{align*}
Multiplying $\eqref{1}_3$ by $\theta-\frac{1}{\overline{\rho_0}|\Omega|}E_0$ and integrating the result over $\Omega$, we infer that
\begin{align*}
  &\frac{1}{2}\frac{d}{dt}\int\rho(\theta-\frac{1}{\overline{\rho_0}|\Omega|}E_0)^2dx
  +\Vert\theta^{\frac{\beta}{2}}\nabla\theta\Vert_{L^2}^2\\
  &=\int(\theta-\frac{1}{\overline{\rho_0}|\Omega|}E_0)\theta^{\alpha}|D(u)|^2dx
  \leq C_0\Vert\nabla u\Vert_{L^2}^2.
\end{align*}
Multiplying the last inequality by $e^{\sigma_2 t}$, one obtains
\begin{align}\label{29}
  &\frac{d}{dt}\big(e^{\sigma_2 t}\Vert\sqrt{\rho}
   (\theta-\frac{1}{\overline{\rho_0}|\Omega|}E_0)\Vert_{L^2}^2\big)
   +2e^{\sigma_2 t}\Vert\theta^{\frac{\beta}{2}}\nabla\theta\Vert_{L^2}^2\nonumber\\
  &\leq \sigma_2e^{\sigma_2 t}\Vert\sqrt{\rho}(\theta-\frac{1}{\overline{\rho_0}|\Omega|}E_0)\Vert_{L^2}^2
   +C_0e^{\sigma_2 t}\Vert\nabla u\Vert_{L^2}^2\nonumber\\
  &\leq 2\sigma_2\tilde{\rho}
   \big(\frac{d}{\pi}+\frac{\tilde{\rho}}{\overline{\rho_0}}\frac{d}{\pi}\big)^2
   \underline{\theta}^{-\beta}e^{\sigma_2 t}\Vert\theta^{\frac{\beta}{2}}\nabla\theta\Vert_{L^2}^2
   +C_0e^{\sigma_2 t}\Vert\nabla u\Vert_{L^2}^2,
\end{align}
owing to Poincar\'{e}'s inequality and
\begin{align*}
  &\Vert\sqrt{\rho}(\theta-\frac{1}{\overline{\rho_0}|\Omega|}E_0)\Vert_{L^2}\\
  &\leq \sqrt{\tilde{\rho}}
   \big(\Vert\theta-\overline{\theta}\Vert_{L^2}
   +\Vert\overline{\theta}-\frac{1}{\overline{\rho_0}|\Omega|}\int\rho\theta dx\Vert_{L^2}
   +\Vert\frac{1}{\overline{\rho_0}|\Omega|}\int\rho\theta dx
   -\frac{1}{\overline{\rho_0}|\Omega|}E_0\Vert_{L^2}\big)\\
  &\leq \sqrt{\tilde{\rho}}
   \big(\frac{d}{\pi}\Vert\nabla\theta\Vert_{L^2}
   +\frac{1}{\overline{\rho_0}|\Omega|}\Vert\int\rho(\theta-\overline{\theta}) dx\Vert_{L^2}
   +\frac{1}{2\overline{\rho_0}|\Omega|}\Vert\int\rho|u|^2dx\Vert_{L^2}\big)\\
  &\leq \sqrt{\tilde{\rho}}
   \big(\frac{d}{\pi}\Vert\nabla\theta\Vert_{L^2}
   +\frac{\tilde{\rho}}{\overline{\rho_0}}\frac{d}{\pi}\Vert\nabla\theta\Vert_{L^2}
   +\frac{\tilde{\rho}}{2\overline{\rho_0}|\Omega|^{\frac{1}{2}}}\Vert u\Vert_{L^2}^2\big)\\
  &\leq \sqrt{\tilde{\rho}}
   \big(\frac{d}{\pi}+\frac{\tilde{\rho}}{\overline{\rho_0}}\frac{d}{\pi}\big)
   \underline{\theta}^{-\frac{\beta}{2}}\Vert\theta^{\frac{\beta}{2}}\nabla\theta\Vert_{L^2}
   +C_0\Vert\nabla u\Vert_{L^2}^2,
\end{align*}
where $\displaystyle \overline{\theta}=\frac{1}{|\Omega|}\int\theta dx$. Besides, choosing $\sigma_2$ such that
\begin{align}\label{30}
  \sigma_2
  =\min\big\{ \, \frac{1}{2}\tilde{\rho}^{-1}\underline{\theta}^{\beta}
   \big(\frac{d}{\pi}+\frac{\tilde{\rho}}{\overline{\rho_0}}\frac{d}{\pi}\big)^{-2}, ~
   \sigma_1\big\},
\end{align}
recalling \eqref{29} and Lemma \ref{L6}, one has
\begin{align}\label{31}
  e^{\sigma_2 t}\Vert\sqrt{\rho}(\theta-\frac{1}{\overline{\rho_0}|\Omega|}E_0)\Vert_{L^2}^2
   +\int_0^t e^{\sigma_2 s}\Vert\theta^{\frac{\beta}{2}}\nabla\theta\Vert_{L^2}^2ds
  \leq C_0.
\end{align}

\underline{step 2}.
Multiplying \eqref{23} by $e^{\sigma_2 t}$, it holds that
\begin{align*}
  &\frac{d}{dt}\big(e^{\sigma_2 t}\Vert\kappa(\theta)\nabla\theta\Vert_{L^2}^2\big)
   +e^{\sigma_2 t}\int\rho\kappa(\theta)|\theta_t|^2dx\\
  &\leq \sigma_2e^{\sigma_2 t}\Vert\kappa(\theta)\nabla\theta\Vert_{L^2}^2
   +\frac{4}{\alpha+\beta+1}\frac{d}{dt}
   \big(e^{\sigma_2 t}\int\theta^{\alpha+\beta+1}|D(u)|^2dx\big)\\
  &\quad+Ce^{\sigma_2 t}\Vert\nabla u\Vert_{H^1}^2
   +Ce^{\sigma_2 t}\Vert\nabla u_t\Vert_{L^2}^2.
\end{align*}
By virtue of Lemma \ref{L6}, \eqref{30} and \eqref{31}, one gets
\begin{align}\label{39}
  e^{\sigma_2 t}\Vert\theta^{\beta}\nabla\theta\Vert_{L^2}^2
   +\int_0^t e^{\sigma_2 s}\Vert\sqrt{\rho}\theta^{\frac{\beta}{2}}\theta_t\Vert_{L^2}^2ds
  \leq C_0.
\end{align}

\underline{step 3}.
Similarly, we have from \eqref{35} that
\begin{align}\label{38}
  &\frac{d}{dt}\big(e^{\sigma_2 t}\Vert\sqrt{\rho}\theta_t\Vert_{L^2}^2\big)
  +e^{\sigma_2 t}\Vert\theta^{\frac{\beta}{2}}\nabla\theta_t\Vert_{L^2}^2\nonumber\\
  &\leq \sigma_2 e^{\sigma_2 t}\Vert\sqrt{\rho}\theta_t\Vert_{L^2}^2
   +C\mathcal{A}(t)e^{\sigma_2 t}\Vert\sqrt{\rho}\theta_t\Vert_{L^2}^2
   +C_0e^{\sigma_2 t}\Vert(\nabla u, u_t)\Vert_{H^1}^2,
\end{align}
where $\mathcal{A}(t)$ is defined in \eqref{37}. Then, combining \eqref{38} with Gr\"{o}nwall's inequality, \eqref{39}, Proposition \ref{P1} and Lemma \ref{L6} gives
\begin{align}\label{41}
  e^{\sigma_2 t}\Vert\sqrt{\rho}\theta_t\Vert_{L^2}^2
   +\int_0^t e^{\sigma_2 s}\Vert\theta^{\frac{\beta}{2}}\nabla\theta_t\Vert_{L^2}^2ds
  \leq C_0.
\end{align}
We thus derive \eqref{40} from \eqref{31}-\eqref{41}, \eqref{21} and \eqref{58}. Consequently, the proof of this lemma is completed.
\end{proof}


\section{Global Well-Posedness}\label{S5}

In this section, we start with the local existence of a strong solution which has been established in \cite{Cho,Guo}.
\begin{lemma}[Local Strong Solution]\label{L5.1}
Suppose that  $(\rho_0, u_0, \theta_0)$ satisfies \eqref{49}-\eqref{79}. Then there exist a
small time $T_0>0$ and a unique strong solution $(\rho, u, \theta, P)$  to the problem \eqref{1}-\eqref{4}
in $\Omega\times[0,T_0]$ satisfying \eqref{45}.
\end{lemma}

With all the a priori estimates in section \ref{S4} and Lemma \ref{L5.1} at hand, we are now in a position to prove Theorem \ref{T1}.

\vspace{0.4cm}

\noindent
\textbf{Proof of Theorem \ref{T1}:}

It follows from \eqref{49}, \eqref{28} and Lemma \ref{L5.1} that there exists a $T_1\in(0,T_0]$ such that \eqref{60} holds for $T=T_1$. Set
\begin{align}\label{46}
  T^{*}\triangleq\big\{T ~\big|~ (\rho,u,\theta,P) ~\text{is a strong solution on} ~ \Omega\times(0,T] ~\text{and} ~ \eqref{60} ~\text{holds}\big\}.
\end{align}
Then, $T^*\geq T_1>0$.

For any $T\in(T_1,T^*]$ with $T$ finite, in view of Proposition \ref{P1}, we know that
\begin{align*}
  \rho\in L^{\infty}(0,T; W^{1,q}), \ \rho_t\in L^{\infty}(0,T; L^q),
\end{align*}
which combining with the work by P.Lions\cite[Lemma 2.3, p.43]{Lions} yields
\begin{align}\label{75}
  \rho\in C([0,T]; W^{1,q}).
\end{align}
Moreover, since
\begin{align*}
  u, \theta\in L^2(0,T; H^3), ~~~~~~
  u_t, \theta_t\in L^2(0,T; H^1)
\end{align*}
and
\begin{align*}
  \Vert u\Vert_{C([0,T]; H^2)}
  \leq C(T)\big(\Vert u\Vert_{L^2(0,T; H^3)}+\Vert u_t\Vert_{L^2(0,T; H^1)}\big)
\end{align*}
(see \cite[Theorem 4, p.304]{Evans}),  one has
\begin{align}\label{76}
  u, \theta\in C([0,T]; H^2).
\end{align}
According to the equations $\eqref{1}_2, \eqref{1}_3$ and \eqref{75}-\eqref{76} (see\cite[p.21]{Zhong2022}), we can see that
\begin{align}\label{77}
  \rho u_t+\rho u\cdot\nabla u, \ \rho \theta_t+\rho \theta\cdot\nabla \theta
  \in C([0,T]; L^2),
\end{align}
we get after using Stokes estimates (Lemma \ref{L03}) that
\begin{align}\label{78}
   P\in C(0,T;H^1).
\end{align}

Now, we claim that
\begin{align}\label{56}
 T^*=\infty.
\end{align}
Otherwise, $T^*<\infty$. Then by Lemmas \ref{L1}-\ref{L5}, the estimate \eqref{65} holds for $T=T^*$. It follows from \eqref{75}-\eqref{78} that
\begin{align*}
  (\rho^*,u^*,\theta^*,P^*)(x)\triangleq(\rho,u,\theta,P)(x,T^*)
  =\lim_{t\rightarrow T^*}(\rho,u,\theta,P)(x,t)
\end{align*}
satisfies
\begin{gather*}
 0\leq\rho^*\in W^{1,q}, ~~~~
 u^*\in H_{0,\sigma}^1\cap H^2, ~~~~
 \underline{\theta}\leq\theta^*\in H_{\mathbf{n}}^2,\\
 \rho^* u^*_t+\rho^* u^*\cdot\nabla u^*\in L^2,   ~~~~
 \rho^* \theta^*_t+\rho^* \theta^*\cdot\nabla \theta^*\in L^2, ~~~~
 P^*\in H^1.
\end{gather*}
And let
\begin{align*}
 & \mathrm{g}_1\triangleq\left\{
        \begin{array}{ll}
          {\rho^*}^{-\frac{1}{2}}(\rho^* u^*_t+\rho^* u^*\cdot\nabla u^*),
           & \text{if} ~ x\in\{x|\rho^*(x)>0\}  \\
          0, & \text{if} ~ x\in\{x|\rho^*(x)=0\}
        \end{array}
      \right. ,\\
 & \mathrm{g}_2\triangleq\left\{
        \begin{array}{ll}
          {\rho^*}^{-\frac{1}{2}}(\rho^*\theta^*_t+\rho^* u^*\cdot\nabla \theta^*),
           & \text{if} ~ x\in\{x|\rho^*(x)>0\}  \\
          0, & \text{if} ~ x\in\{x|\rho^*(x)=0\}
        \end{array}
      \right. .
\end{align*}
Then, from Proposition \ref{P1}, we obtain that $\mathrm{g}_1, \mathrm{g}_2\in L^2$ and satisfy the compatibility conditions
\begin{align*}
  \left\{
    \begin{array}{rr}
     -\dive(2\mu(\theta^*)D(u^*))+\nabla P^*=\sqrt{\rho^*}\mathrm{g}_1, \\
     -\dive(\kappa(\theta^*)\nabla\theta^*)-2\mu(\theta^*)|D(u^*)|^2=\sqrt{\rho^*}\mathrm{g}_2.
    \end{array}
  \right.
\end{align*}
Hence, we can take $(\rho^*,u^*,\theta^*)$ as the initial data and apply Lemma \ref{L5.1} to extend the strong solution beyond $T^*$. This contradicts the definition of $T^*$ in \eqref{46}. Therefore, $T^*=\infty$.

Finally, to finish the proof of Theorem \ref{T1}, it remains to prove \eqref{47}-\eqref{55}. In fact, the decay estimates on $u$ and $\theta$ have been established in Lemmas \ref{L6} and \ref{L7}, respectively. Collecting \eqref{42}, \eqref{40} and \eqref{56}, we complete the proof of Theorem \ref{T1}.
\hfill $\qedsymbol$

\begin{remark}\label{r3}
It is well known that  $\rho$ is a constant along the streamline (flow map). Thus, $\rho$ cannot converge to any constant as $t\rightarrow\infty$.
Indeed, $\forall \, x\in\Omega, t\in[0,\infty]$, we define $X(x,s;t)$ as follows
\begin{align}
\left\{
  \begin{array}{l}
    \displaystyle
    \frac{d}{ds}X(x,s;t)=u(X(x,s;t),s), ~~~\forall \, s\in[0,\infty],\\[2mm]
    X(x,t;t)=x.
  \end{array}
\right.
\end{align}
Using the fact that $u\in L^1(0,\infty; W^{1.\infty})$ and Cauchy-Lipschitz theorem (Lemma Appendix A.1. in \cite{Choi}), we know that $X(x,s;t)$ is well-defined. Besides, we obtain from \eqref{1} that
\begin{align*}
  \frac{d}{ds}\rho(X(x,s;t),s)
  =0,
\end{align*}
which  means for any $(x,s,t)\in\Omega\times[0,\infty]\times[0,\infty]$,
\begin{align*}
  \rho(X(x,s;t),s)=\rho(X(x,t;t),t)=\rho(x,t).
\end{align*}
Taking $s=0$, we get
\begin{align*}
  \rho(x,t)=\rho_0(X(x,0;t)), ~~~
   \forall~ (x,t)\in\Omega\times[0,\infty].
\end{align*}
We see that
\begin{align}
 \rho(x,\infty)
 \triangleq\lim_{t\rightarrow\infty}\rho(x,t)=\rho_0(X(x,0;\infty)).
\end{align}
Therefore, the state of $\rho$ at infinity time is determined by $X(x,0;\infty)$, which also depends on the velocity $u$. Consequently, in general it cannot converge to a constant unless $\rho_0\equiv const.$.
\end{remark}

\section*{Acknowledgement}

The work is  supported   by the Fundamental Research Funds for the Central Universities, CHD (No.300102122115).

\section*{Conflict of interest statement}
The authors declare that there is no conflict of interests regarding the publication of this
article.

\end{document}